\documentclass{amsart}
\usepackage{amsmath, amssymb, graphicx, epsfig, verbatim, psfrag, color,amscd,stmaryrd,rsfso,mathabx} 
\usepackage{tikz}
\usepackage{amsthm}
\usepackage{dutchcal}

\usepackage{enumitem}
\usetikzlibrary{arrows}
\usepackage[T1]{fontenc}
\usepackage[utf8]{inputenc}

\usepackage{mathrsfs}
\newcommand\oL{\vec{L}}
\newcommand\DiagUp{\widehat{\Diag}}

\newcommand\wpt{\mathrm{w}}
\newcommand\zpt{\mathrm{z}}

\newcommand\gen{K}
\newcommand\ind{\mathrm{ind}}
\newcommand\longchord{\upsilon}

\newcommand\DmodAlg{\Dmod^{alg}}

\newcommand\MatchIn{\widecheck{\Matching}}
\newcommand\Win{\widecheck{W}}
\newcommand\Mout{\widehat{\Matching}}
\newcommand\Zin{\widecheck{Z}}
\newcommand\Zmid{Z^{\parallel}}
\newcommand\Zout{\widehat{Z}}
\newcommand\Mmid{M^\parallel}
\newcommand\Wmid{W^{\parallel}}
\newcommand\Wup{\widecheck{W}}

\newcommand\Hmid{{\mathcal H}^\parallel}
\newcommand\HmidEx{{\mathcal H}^{\parallel,x}}
\newcommand\alphaOut\alphaout
\newcommand\alphain{\widecheck{\alpha}}
\newcommand\alphaIn\alphain

\newcommand\alphaout{\widehat{\alpha}}

\newcommand\Clgin{\widecheck{\Clg}}
\newcommand\Clgout{\widehat{\Clg}}
\newcommand\cClgin{\Clgin^\star}
\newcommand\cClgout{\Clgout^\star}
\newcommand\BigBlg{\mathfrak{B}}
\newcommand\cBigBlg{\BigBlg^\star}
\newcommand\Blgin{\widecheck{\Blg}}

\newcommand\Blgout{\widehat{\Blg}}

\newcommand\DAmodExtHat{{\widehat{\DAmod}}^x}
\newcommand\DAmodEx{\DAmod^x}

\newcommand\HmidE{{\mathcal H}^{\parallel,e}}

\newcommand\rhos{\boldsymbol{\rho}}

\newcommand\UnparModFlow{\widehat{\mathcal M}}
\newcommand\alphas{\boldsymbol{\alpha}}
\newcommand\betas{\boldsymbol{\beta}}
\newcommand\Source{\mathscr{S}}

\newcommand\orb{o}
\newcommand\goesto{\mapsto}

\newcommand\Hup{\mathcal{H}^{\wedge}}
\newcommand\Hdown{\mathcal{H}^{\vee}}

\usetikzlibrary{matrix,arrows,positioning}
\tikzset{cdlabel/.style={above,sloped,
    execute at begin node=$\scriptstyle,execute at end node=$}}
\tikzset{algarrow/.style={->, thick}}
\tikzset{blgarrow/.style={->, thick}}
\tikzset{clgarrow/.style={->, thick}}
\tikzset{tensoralgarrow/.style={double, double equal sign distance, -implies}}
\tikzset{tensorblgarrow/.style={double, double equal sign distance, -implies}}
\tikzset{tensorclgarrow/.style={double, double equal sign distance, -implies}}
\tikzset{modarrow/.style={->, dashed}}
\tikzset{othmodarrow/.style={->, thick}}
\tikzset{Amodar/.style={->, dashed}}
\tikzset{Dmodar/.style={->, dashed}}

\newcommand\HD{\mathcal H}

\newcommand\Diag{\mathcal D}
\newcommand\States{\mathfrak{S}}

\def\endproof{\relax\ifmmode\expandafter\endproofmath\else
  \unskip\nobreak\hfil\penalty50\hskip.75em\hbox{}\nobreak\hfil\bull
  {\parfillskip=0pt \finalhyphendemerits=0 \bigbreak}\fi}
\def\endproofmath$${\eqno\bull$$\bigbreak}
\def\bull{\vbox{\hrule\hbox{\vrule\kern3pt\vbox{\kern6pt}\kern3pt\vrule}\hrule}}

\newcommand\CanonDD{\mathcal K}

\newcommand\Lmid{\ell_\parallel}

\newcommand\HFKa{\widehat{\mathit{HFK}}}
\newcommand\HFLa{\widehat{\mathit{HFL}}}
\newcommand\HFLm{{\mathit{HFL}}^-}
\newcommand\CFLa{\widehat{\mathit{CFL}}}

\newtheorem{thm}{Theorem}[section]

\newtheorem{cor}[thm]{Corollary}

\newtheorem{lemma}[thm]{Lemma}
\newtheorem{prop}[thm]{Proposition}
\newtheorem{defn}[thm]{Definition}

\newtheorem{rem}[thm]{Remark}
\newtheorem{remark}[thm]{Remark}

\numberwithin{equation}{section}

\newcommand\OneHalf{\frac{1}{2}}

\newcounter{bean}

\newcommand\IdempRing{I}

\newcommand\Idemp[1]{\mathbf{I}_{#1}}
\newcommand\DT{\boxtimes}
\newcommand\x{\mathbf x}

\newcommand\y{\mathbf y}

\newcommand\lsup[2]{^{#1}{#2}}

\newcommand\cBlg{\Blg^\star}

 \newcommand{\Z}{\mathbb Z}   \newcommand{\R}{\mathbb R}

\newcommand\Ax{\mathbf A}

\newcommand\BB{\mathbf B}
\newcommand\XX{\mathbf X}
\newcommand\YY{\mathbf Y}

\newcommand\Field{\mathbb F}

\newcommand\Alg{\mathcal A}
\newcommand\Blg{\mathcal B}
\newcommand\BlgZ{{\mathcal B}_0}
\newcommand\Clg{\mathcal C}

\newcommand\Ainf{{\mathcal A}_{\infty}}
\newcommand\Ainfty\Ainf
\newcommand\Zmod[1]{{\mathbb Z}/{#1}{\mathbb Z}}

\newcommand\OurRing{\mathcal R}

\newcommand\Amod{Q}
\newcommand\Dmod{R}
\newcommand\DmodHat{\widehat \Dmod}
\newcommand\AmodHat{\widehat \Amod}
\newcommand\DAmodHat{\widehat \DAmod}

\newcommand\DAmod{RQ}

\newcommand\Mdown{\widecheck{M}}
\newcommand\Mup{\widehat{M}}

\newcommand\doms{\mathcal D}

\newcommand\Mgr{\mathbf{m}}
\newcommand\Agr{\mathfrak{A}}
\newcommand\weight{\mathbcal{w}}

\newcommand\DuAlg{{\mathcal A}'}

\newcommand\Matching{M}

\newcommand\cClg{\Clg^\star}
\newcommand\nDuAlg{\Alg''}

\newcommand\Iup{\widehat{\mathbf I}}

\newcommand\MarkedMin{V}

\newcommand\MarkedMinHat{\widehat{\MarkedMin}}

\newcommand\bOut{\widehat{b}}
\newcommand\bIn{\widecheck{b}}

\newcommand\Alex{\mathbb A}
\newcommand\Mas{\mu}
\newcommand\CFL{\mathit{CFL}}
\newcommand\HFL{\mathit{HFL}}
\newcommand\dIn{\widecheck{\partial}}
\newcommand\dOut{\widehat{\partial}}

\newcommand\nMarkedMin{\MarkedMin'}
\newcommand\nMarkedMinHat{\MarkedMinHat'}
\newcommand\cald{\mathcal D}

\makeatletter
\renewenvironment{proof}[1][\proofname]{\par
\pushQED{\qed}%
\normalfont \topsep6\p@\@plus6\p@\relax
\trivlist
\item\relax
{\bf#1\@addpunct{.}}\hspace\labelsep\ignorespaces
}{%
\popQED\endtrivlist\@endpefalse
}
\makeatother

\begin{document}
\title{Algebras with matchings and link Floer homology}

\author[Peter S. Ozsv\'ath]{Peter Ozsv\'ath}
\thanks {PSO was supported by NSF grant number DMS-1405114 and DMS-1708284.}
\address {Department of Mathematics, Princeton University\\ Princeton, New Jersey 08544} 
\email {petero@math.princeton.edu}

\author[Zolt{\'a}n Szab{\'o}]{Zolt{\'a}n Szab{\'o}}
\thanks{ZSz was supported by NSF grant numbers DMS-1606571 and DMS-1904628}
\address{Department of Mathematics, Princeton University\\ Princeton, New Jersey 08544}
\email {szabo@math.princeton.edu}

\begin{abstract}
  We explain how to use bordered algebras to compute
  a version of link Floer homology.
  As a corollary, we also give a fast computation of the Thuston polytope for
  links in $S^3$.
\end{abstract}

\maketitle

\newcommand\OneFourth{\frac{1}{4}}
\newcommand\lk{\ell k}
\newcommand\orL{\vec{L}}
\newcommand\Ta{\mathbb{T}_\alpha}
\newcommand\Tb{\mathbb{T}_\beta}
\newcommand\DmodRed{\overline{\Dmod}}
\section{Introduction}
\label{sec:Intro}

The aim of this paper is to generalize the bordered construction of
knot invariants from~\cite{HolKnot} to handle  the case of
links, giving a practical computation of a variant of link Floer
homology for links in $S^3$~\cite{Links}, which is sufficient to
determine the Thurston polytope of such links. (The reader should
compare with the computation of link homology using grid diagrams
from~\cite{MOS}; see also~\cite{MOST,GridBook}.)

As background, let $\orL$ be an oriented link with $\ell$
components. Such a a link has a multi-variable Alexander polynomial
$\Delta_{L}\in \Z[t_1^{\pm 1/2},\dots,t_\ell^{\pm 1/2}]$.  The
oriented meridians for the link give an identification of $\Z[t_1^{\pm
  1},\dots,t_{\ell}^{\pm 1}]$ with the group-ring $\Z[H^1(S^3\setminus
L)]$; and in this way, we can view the Alexander polynomial
$\Delta_{L}$ as defining for an (unoriented) link an element of
$\Z[H^1(S^3\setminus L)]$.  As explained in~\cite{Kauffman}, this
polynomial can be written as a state sum of Kauffman states for a
decorated link projection, with local contributions as shown in
Figure~\ref{fig:KauffLink}. 

\begin{figure}[h]
 \centering
 \input{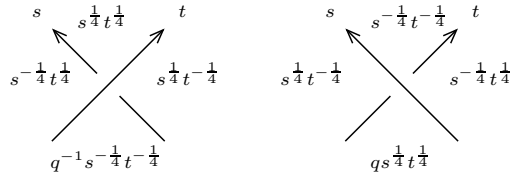}
 \caption{{\bf Kauffman states for links.}  Monomial contributions at
   each crossing of a Kauffman state. Monomials are in the variable
   $q$, whose exponent records the Maslov grading, and variables
   corresponding to the oriented strands, whose exponents record the
   Alexander gradings. The strand exiting on the top left resp. top
   right corresponds to the variable $s$ resp. $t$.}
 \label{fig:KauffLink}
 \end{figure}

\begin{rem}
  Note that there are two conventions in link (and knot) Floer
  homology, corresponding how the orientation of a link is encoded in
  a pointed Heegaard diagram. Since link Floer homology is invariant
  under reversing the orientation of all components, both conventions
  ultimately lead to the same invariant. Figure~\ref{fig:KauffLink}
  is consistent with the
  orientation convention from~\cite{GridBook}; it is opposite to
  the one from~\cite[Figure~1]{BorderedKnots}, which in turn follow the
  conventions from~\cite{ClassKnots}.
\end{rem}

In~\cite{Links}, we defined a version of ``link Floer homology'',
$\HFLa(L)$, which is a finite-dimensional vector space over
$\Field=\Zmod{2}$.  That group is equipped with two gradings, one with
values in $\Z$, and another with values in an affine space
${\mathbb H}\subset \OneHalf H_1(S^3\setminus L;\Z)$
for $H^1(S^3\setminus L;\Z)$. The subset ${\mathbb H}\subset \OneHalf H_1(S^3\setminus L;\Z)$
consists of elements $\sum_{i=1}^{\ell} a_i \cdot [\mu_i]$ with
\[ 2 a_i +\lk (L_i,L\setminus L_i)\in 2\Z, \]
for $i=1,\dots,\ell$; where $L_i$ denotes the $i^{th}$ component of $L$.
We denote the direct sum decomposition of link Floer homology
\[ \HFKa(L)\cong \bigoplus_{d\in\Z, h\in {\mathbb H}} \HFKa_d(L,h).\]

Link Floer homology with $\ell>1$ has two key features. One is its relationship
with the Alexander polynomial:
\[ \bigoplus_{d\in \Z,h\in {\mathbf H}} (-1)^{d} \dim(\HFLa_d(L,h))
[h] = \left(\prod_{i=1}^{\ell} (\mu_i^{1/2}-\mu_i^{-1/2})\right)\cdot
\Delta_L,\] 
where $\mu_i$ denotes the oriented meridian of the link component $L_i$. (See~\cite[Equation~(1)]{Links}.)  Another is its
relationship with the Thurston polytope.  This is stated in terms of
the {\em link Floer homology polytope} in $H^1(S^3\setminus L;\R)$,
which is the convex hull of all $h\in {\mathbb H}$ with
$\HFLa(L,h)\neq 0$. The Minkowski sum of the dual Thurston polytope
with the symmetric hypercube in $H^1(S^3\setminus L)$ with edge-length
two is twice the link Floer homology polytope.
(See~\cite[Theorem~1.1]{LinkTN}; see also~\cite{YiNiGenus}.)

We will consider diagrams $\Diag$ for the projection of 
an oriented link $\orL$. These
diagrams are drawn on the $xy$ plane, and we assume that they are
generic in the following sense:
\begin{itemize}
\item the restriction of $y$ is a Morse function,
  with at most one critical point for each $y$ value
\item the $y$-values of all crossings are distinct from each other
  and from the $y$-values of each critical point.
\end{itemize}

A {\em marked link diagram} $\Diag$ is a projection of an oriented
link $\orL$ together with a collection of basepoints, called {\em
  markers}, one on each component of $L$, and one of these markers is
distinguished.  (The distinguished marker is indicated by a star, and
the others are indicated by a dot.) A marked link projection has two
distinguished regions, which are the regions adjacent to the
distinguished basepoint.

We call a marked link projection {\em canonically marked} if the
marking on each component of the link is the global minimum of the
height function restricted to that component; and the distinguished
marker is at the global minimum of the height function on the entire
projection. A {\em marked upper link diagram} is the restriction of a marked
link projection to an upper halfplane. The upper link diagram is
called {\em canonically marked} if only the closed components of the
tangle have markings on them, and those markings occur at the global
minima of the height function restricted to the closed components.  If
we slice a canonically marked link diagram along a generic horizontal
slice, the diagram falls into a canonically marked upper diagram (and
a lower diagram).  See Figure~\ref{fig:MarkedLink} for some examples.

\begin{figure}[h]
 \centering
 \input{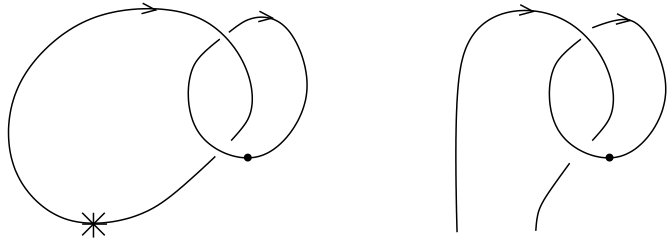}
 \caption{{\bf Marked link projection.}  At the left, a canonically
   marked link projection for the Hopf link.  At the right, a
   canonically marked upper link diagram.}
 \label{fig:MarkedLink}
 \end{figure}

In the present paper, we define a type $D$ structure
$\DmodHat({\Diag})$ to each canonically marked upper link diagram. For
a link in bridge position, where each marker is adjacent to the one of
the two distinguished regions, the generators have an interpretation
in terms of Kauffman states; see
Section~\ref{subsec:HeegLinkProj}.

We also give a method for computing $\DmodHat(\Diag)$.  Specifically,
decompose $\Diag$ into elementary pieces, cutting along horizontal
slices, so that each piece consists of either a maximum, a crossing
between adjacent strands, or a minimum, which can be marked or
not. After introducing extra crossings if necessary, we can assume all
the marked minima occur at the left of the diagram (i.e. they occur
between the between the first
two strands). To each elementary piece, we associate a type $DA$
bimodule, so that the $\DmodHat(\Diag)$ is obtained as an iterated
tensor product of the pieces. Indeed, to the pieces containing
crossings, maxima, and unmarked minima the DA bimodules were already
described in~\cite[Section]{HolKnot}.

Thus, the technical core of this paper is to compute the DA bimodule
of a marked minimum (connecting the first two strands). With this
computation, a suitable adaptation of the pairing theorem
(\cite[Theorem~\ref{HK:thm:PairAwithD}]{HolKnot}; see
also~\cite[Theorem~1.3]{InvPair}) completes the computation of
$\DmodHat(\Diag)$.

Finally, the relationship with link Floer homology is given as
follows. Suppose that $\Diag$ is a canonically marked link diagram.
Cutting it above the global minimum, we obtain a type $D$ structure
$\DmodHat(\Diag)$ over the algebra $\Field[U_1,U_2]/U_1 U_2=0$.

\begin{thm}
  \label{thm:ComputeHFLa}
  Let $\Diag$ be a canonically marked diagram representing $\orL$, and
  $\Diag^+$ be the upper diagram immediately above the global minimum.
  $\HFLa(\orL)$ is the homology of the chain complex obtained from 
  $\DmodHat(\Diag)$ by specializing to $U_1=U_2=0$.
\end{thm}

There are other versions of link Floer homology. One other variant is
a module over a polynomial algebra $\Field[U]$, associated to an
oriented link $\orL$, whose $U=0$ specialization gives $\HFLa(L)$.
This variant, and indeed some further enhancements of it, can also be
computed by our techniques.  The key point here is to develop versions
of the bimodule associated to a marked minimum, with more algebraic
structure. 

This paper is organized as follows. In Section~\ref{sec:Heegaard} we
describe the Heegaard diagrams relevant to this paper. We start by
recalling Heegaard diagrams for links, following~\cite{Links}.  We
generalize these notions, defining Heegaard diagrams associated to
marked upper diagrams. These are a slight generalization of the
Heegaard diagrams considered in~\cite{HolKnot}: the novelty here is
that now we allow for closed components, provided that they are
marked. This section also contains the corresponding generalization
for middle diagrams.  Finally, we describe the Heegaard diagram
associated to a marked link projection generalizing slightly the
Heegaard diagram of a knot projection as defined in~\cite{ClassKnots}.

In Section~\ref{sec:HFL}, we
recall various versions of link Floer homology which can be computed
using our methods. (The most general version we compute here
depends also on an additional choice of distinguished link component.)

In~\cite{HolKnot}, upper Heegaard diagrams give rise to curved type
$D$ structures over an algebra $\Clg$, while the middle diagrams can
be extended to bimodules over a larger algebra $\Blg$. We review
notation in Section~\ref{sec:Algebra}.

In Section~\ref{sec:DAmods}, we explain how to generalize the
holomorphic constructions from~\cite{HolKnot} to associate modules to
marked Heegaard diagrams.  In Section~\ref{subsec:PairingTheorem}, we
adapt the pairing theorem from~\cite{HolKnot} to the case of marked
diagrams.

In Section~\ref{sec:AlgMarkedMin}, we algebraically
define $DA$ modules, which are associated to marked minima. 
In Section~\ref{subsec:ComputeMarkedMin}, we verify that the algebraic
construction indeed agrees with the construction defined using
holomorphic methods.

In Section~\ref{sec:ComputeHFL}, we assemble the pieces to compute
link Floer homology. We obtain Theorem~\ref{thm:ComputeHFLa} as a corollary of
an algebraically enhanced version,
Theorem~\ref{thm:GeneralComputation}.

{\bf Acknowledgements.} The authors wish to thank Nate Dowlin, Robert
Lipshitz, and Andy Manion for helpful conversations.

\section{Link Floer homology}
\label{sec:HFL}

Let $L$ be an $\ell$-component link. Choose an orientation on each
component of $L$, and choose also a distinguished component. We denote
this date $\orL^\star$.  We find it convenient to label the components
$\{L_i\}_{i=1}^{\ell}$ of $L$ so that $L_1$ is the distinguished
component. 

Consider polynomial algebra whose generators are labelled
$w_1,z_1,\dots,w_\ell,z_\ell$, where we think of $w_i$ as
corresponding to the link component $L_i$ with its given orientation
and $z_i$ as corresponding to $L_i$ with its opposite orientation.

Specializing the construction from~\cite{Links} (which we recall in detail 
in Section~\ref{sec:Heegaard} below), 
there is a version of link Floer homology, which is a chain complex over
the ring
\[ \OurRing=\Field[w_1,z_1,\dots w_{\ell},z_{\ell}]/w_1 z_1=0.\]
Specifically, let $\HD$ be a $2\ell$ pointed Heegaard diagram
representing $\orL^\star$, equipped with $2\ell$ basepoints labelled
$(\wpt_1,\zpt_1,\dots,\wpt_\ell,\zpt_\ell)$ (so that $\wpt_i$, $\zpt_i$ represent
$L_i$).  We can form $\CFL(\orL^\star)$ the free $\OurRing$-module
generated by Heegaard
states, and with differential determined by
\[ \partial \x = \sum_{\y\in\Ta\cap\Tb} \sum_{\{\phi\in W(\x,\y)} \#\UnparModFlow(\phi)
\left(\prod_{i=1}^{\ell} w_i^{n_{\wpt_i}(\phi)} z_i^{n_{\zpt_i}(\phi)}\right)\cdot \y,\]
where $\UnparModFlow(\phi)$ is a moduli space of pseudo-holomorphic Whitney disks
in the homotopy class specified by $\phi$.
Let $\HFL(\orL^\star)$ denote the homology of this chain complex,
viewed as a module over $\OurRing$. 
We define relative gradings so that
if $\phi\in\pi_2(\x,\y)$, then
\begin{align*}
  \Alex(\x)-\Alex(\y)&=(n_{\zpt_1}(\phi)-n_{\wpt_1}(\phi),\dots,n_{\zpt_\ell}(\phi)-n_{\wpt_\ell}(\phi) \\
  \Mgr(\x)-\Mgr(\y)&=\Mas(\phi)-2 \sum_{\wpt_i}n_{\wpt_i}(\phi). 
\end{align*}
To give an absolute Maslov grading, we require that the specialization
of the complex to $z_1=\dots=z_\ell=1$, which has homology isomorphic
to $\Field[U]$ (where each $w_i$ acts as multiplication by $U$),
should have its generator supported in Maslov grading equal to zero.
Equivalently, if we set $w_1=\dots=w_\ell=0$ and $z_1=\dots=z_\ell=1$,
and take the homology of the resulting complex, then we obtain a
graded group which contains a non-zero, homogenous element of Maslov
grading $0$, and no homogeneous generators with positive Maslov
grading.

An absolute Alexander grading can also be specified by a certain symmetry of grid homology; cf. Equation~\eqref{eq:Symmetry} below.

The link complex above can be thought of as giving a version of the
knot Floer homology of the three-manifold obtained as large surgery on
the distinguished component. Dependence of the construction on the
distinguished component can be removed in various algebraic
specializations. For instance, one could set $w_i z_i=0$ for all
$i=1,\dots,\ell$ in $\CFL$, and then take homology to obtain an
invariant of the underlying oriented link $\orL$. Or one could form the specialization with
$z_1=\dots=z_\ell=0$. This invariant is referred to as {\em unblocked
  grid homology} (in the context of grid diagrams)
in~\cite[Chapter~11]{GridBook}.  Specializing further to
$w_1=\dots=w_\ell$ (i.e. with $z_1=\dots=z_\ell=0$) gives the {\em
  collapsed grid homology} in the terminology
of~\cite[Chapter~8.2]{GridBook}, which we denote here
$\HFLm(\orL)$. Finally, setting all $w_i=0=z_i$ for all
$i=1,\dots,\ell$ gives the complex for computing $\HFLa(\orL)$ from the
introduction (the {\em simply blocked} grid homology, in the
terminology of~\cite{GridBook}).

These specializations are perhaps more natural objects than
$\HFL(\orL^\star)$; and indeed we will typically consider the case of
$\HFLa$ and its bordered analogues as warm-ups; but we will also
consider $\HFL(\orL^\star)$, since it is the algebraically most
general construction that we can compute using our present methods.

Finally, we recall the following useful a symmetry in link Floer
homology. To this end, let $\orL$ be an oriented link, and $\orL'$ be
the oriented link obtained by reversing the orientation of the
$i^{th}$ component of $\orL$. We have the symmetry
\begin{equation}
  \label{eq:Symmetry}
  \HFLa_d(\orL,(s_1,\dots,s_\ell))\cong
  \HFLa_{d-2s_i+\kappa_i}(\orL',(s_1,\dots,s_{i-1},-s_i,s_{i+1},\dots,s_\ell)
\end{equation}
where
\[ \kappa_i=\lk (L_i,L\setminus L_i)\in 2\Z. \]
(cf.~\cite[Proposition~11.4.2]{GridBook}). This identification removes
the additive indeterminacy of the Alexander grading.

\newcommand\basvec[1]{{\mathrm e}_{#1}}
\newcommand\ws{\mathbf{w}}
\newcommand\zs{\mathbf{z}}
\newcommand\Lup{\widehat{\ell}}
\newcommand\Lin{\Ldown}
\newcommand\Oin{\Odown}
\newcommand\Xin{\Xdown}
\newcommand\Oup{\widehat{\wpt}}
\newcommand\Xdown{\widecheck{\zpt}}
\newcommand\Odown{\widecheck{\wpt}}
\newcommand\Omid{\wpt^{\parallel}}
\newcommand\Xmid{\zpt^{\parallel}}
\newcommand\Xup{\widehat{\zpt}}
\newcommand\Ldown{\widecheck{\ell}}
\section{Heegaard diagrams and marked link projections}
\label{sec:Heegaard}

In this section, we describe the Heegaard diagrams that will be used
for marked links.  In Section~\ref{subsec:HeegLinks}, we recall the
Heegaard diagrams for oriented links from~\cite{Links}.  In
Section~\ref{subsec:MarkedUpperHeegs}, we generalize these to marked
upper diagrams, which we further generalize in Section~\ref{subsec:MarkedMid} to marked middle diagrams.  In
Section~\ref{subsec:MarkedLower}, we describe marked lower diagrams.
In Section~\ref{subsec:HeegLinkProj} we describe the Heegaard diagrams
for marked link projections, generalizing the construction
from~\cite{ClassKnots}.

\subsection{Heegaard diagrams for links}
\label{subsec:HeegLinks}

Recall~\cite[Definition~3.7]{Links} that an oriented link $\orL$ can be represented by a Heegaard
diagram, 
\[ (\Sigma_0,\alphas,
\betas,\wpt_1,\zpt_1,\dots,\wpt_\ell,\zpt_\ell),\]
where:
\begin{enumerate}[label=(L-\arabic*),ref=(L-\arabic*)]
  \item $\Sigma_0$ is a genus $g$ surface with $2\ell$ boundary components,
    labelled
    \[ \wpt_1,\zpt_1,\dots,\wpt_\ell,\zpt_\ell.\]
  \item $\alphas=\{\alpha_1,\dots,\alpha_{g+\ell-1}\}$ 
    is a set of pairwise disjoint,
    embedded curves.
  \item 
    \label{A-corr}
    The surface obtained by cutting $\Sigma_0$ along
    the $\alpha$-curves has $\ell$ components
    $A_1,\dots,A_{\ell}$, each of which contains exactly one
    $\wpt$-marked boundary component and one $\zpt$-marked boundary component.
  \item $\betas=\{\beta_1,\dots,\beta_{g+\ell-1}\}$ is another
    set of  pairwise disjoint,
    embedded curves.
  \item
        \label{B-corr}
        The surface obtained by cutting $\Sigma_0$ along
        the $\beta$-curves consists of $\ell$ components
        $B_1,\dots,B_{\ell}$,
        each of which contains exactly one $\wpt$-marked boundary component
        and one $\zpt$-marked boundary component.
      \item The two one-to-one correspondences between the $\wpt$- and
        $\zpt$-boundary components specified in Parts~\ref{A-corr}
        and~\ref{B-corr} coincide; i.e. if we can label all the
        boundary components and connected components so that $A_i$ and
        $B_i$ contain $\wpt_i$ and $\zpt_i$.
\end{enumerate}
The underlying three-manifold (which in our case will be $S^3$) is
obtained by filling in the boundary components of $\Sigma_0$ with
disks to obtain a Heegaard surface $\Sigma$, equipped with attaching
circles $\alphas$ and $\betas$. This Heegaard diagram $(\Sigma,\alphas,\betas)$
specifies the three-manifold.
Thinking  of the $\wpt$ and
$\zpt$-boundaries as inducing corresponding marked points in $\Sigma$, we
can obtain a link by first connecting the $\wpt$ and $\zpt$-markings by
pairwise disjoint, embedded arcs in $\Sigma$ that are disjoint from
the $\alphas$, and pushing it slightly into the $\alpha$-handlebody;
and then connecting the $\wpt$ and $\zpt$-markings analogously in the
$\beta$-handlebody. An orientation on this link is specified by
demanding that the portion in the $\alpha$-handlebody is oriented as
arcs from $\wpt$ to $\zpt$.

(Our $\wpt$-markings $\{\wpt_1,\dots,\wpt_\ell\}$ resp. $\zpt$-markings were
denoted $\ws$ resp. $\zs$ in~\cite{Links}; 
the $\wpt$ and $\zpt$-markings correspond to the $O$-markings and $X$-markings respectively for grid diagrams~\cite{GridBook}. 
The convention for the induced orientation on the link
here is opposite to the one given in~\cite{Links}, but it is consistent
with the one given for grid diagrams in~\cite{GridBook}.)

We will consider Heegaard diagrams satisfying the following property.

\begin{defn}
  \label{def:PerDom}
  A {\em periodic domain} $P$ is a two-chain in $\Sigma_0$ with
  \[ \partial P = \sum_{i} m_i [\alpha_i] + \sum_{j} n_j[\beta_j],\]
  where $m_i,n_i\in\Z$.
  A periodic domain is called {\em somewhere
    positive} (resp. {\em negative}) if some local multiplicity is
  positive (resp. negative).
  The  Heegaard diagram is called {\em
    admissible} if every non-zero periodic domain is somewhere positive
  (and hence also somewhere negative).
\end{defn}

Note that we can think of these as two-chains in $\Sigma$ with
vanishing local multiplicity at the marked points corresponding to the
various $\wpt_i$ and the $\zpt_i$.

A {\em Heegaard state} for a $(\Sigma_0,\alphas,
\betas,\wpt_1,\zpt_1,\dots,\wpt_\ell,\zpt_\ell)$ is a $g+\ell-1$-tuple of points
$\{x_1,\dots,x_{g+\ell-1}\}$ with $x_i\in \alpha_{\sigma(i)}\cap
\beta_{i}$, where $\sigma$ is an element of the symmetric group on
$g+\ell-1$ letters.

\subsection{Marked upper Heegaard diagrams}
\label{subsec:MarkedUpperHeegs}

We modify the notion of upper Heegaard diagram
from~\cite[Section~\ref{HK:sec:Heegs}]{HolKnot} to include closed components.

\begin{defn}
  \label{def:MarkedUpperDiagram}
A {\em marked upper Heegaard diagram} is 
the following data:
\begin{itemize}
\item a surface $\Sigma_0$ of genus $g$ and
  $2n$ boundary components, labelled $Z_1,\dots,Z_{2n}$,
  and $2 \Lup$ additional boundary components labelled 
  $\Oup_1,\Xup_1,\dots,\Oup_{\Lup},\Xup_{\Lup}$.
\item a collection
  of disjoint, embedded arcs $\{\alpha_i\}_{i=1}^{2n-1}$, so that
  $\alpha_i$ connects $Z_i$ to $Z_{i+1}$,
\item a collection of
  disjoint embedded closed curves $\{\alpha^c_i\}_{i=1}^{g+\Lup}$
  (which are also disjoint from $\alpha_1,\dots,\alpha_{2n-1}$),
\item 
  another collection of embedded, mutually disjoint closed curves
  $\{\beta_i\}_{i=1}^{g+n+\Lup-1}$.  
\end{itemize}
We require this data to also satisfy the following properties:
\begin{enumerate}[label=(UD-\arabic*),ref=(UD-\arabic*)]
\item For each $i\in\{1,\dots,2n-1\}$, $j\in \{1,\dots,g+\Lup\}$, and $k\in\{1,\dots,g+\Lup+n-1\}$,
$\alpha_i$ and $\alpha_j^c$ curves are transverse to $\beta_k$.
\item
Both sets of $\alpha$-and the
$\beta$-circles consist of homologically linearly
independent curves (in $H_1(\Sigma_0)$).
\item 
  \label{UD:TwoBoundariesApieceB}
  Each component $B_1,\dots,B_{n+\Lup}$ of the surface obtained by cutting
  $\Sigma_0$ along $\beta_1,\dots,\beta_{g+\Lup+n-1}$, is required to
  contain exactly two boundary components, which are either both of type $Z$,
  or one is 
  of type $\zpt$ and the other is of type $\wpt$.
\item
  \label{UD:TwoBoundariesApieceA}
  Each component $A_1,\dots,A_{\Lup+1}$ of the the surface obtained by
  cutting $\Sigma_0$ along $\alpha^c_1,\dots,\alpha^c_{g+\Lup}$, is
  required to either contain all the $Z$-components, or exactly one of
  the $\wpt$-components and one of the $\zpt$-components.
\item
  \label{UD:LComponentLink}
  Condition~\ref{UD:TwoBoundariesApieceA} gives a one-to-one
  correspondence between the $\zpt$-boundaries and the $\wpt$-boundaries;
  Condition~\ref{UD:TwoBoundariesApieceB} gives another one-to-one
  correspondence between the $\zpt$-boundaries and the $\wpt$-boundaries.
  We require that these two correspondences coincide.
\item
  \label{UD:MinPerDom}
  Together, the closed curves
  $\{\alpha_i^c\}_{i=1}^{g+\Lup}$ and
  $\{\beta_i\}_{i=1}^{g+\Lup+n-1}$
  span a $2g+\Lup+n-1$-dimensional subspace of $H_1(\Sigma_0;\Z)$.
\end{enumerate}
\end{defn}

\begin{rem}
  An upper diagram specifies a three-manifold-with-boundary $Y$
  equipped with an $\Lup$-component, oriented link, and whose boundary
  is a sphere containing an embedded collection of $2n$ arcs.
  (Condition~\ref{UD:LComponentLink} ensures that the oriented link
  indeed has $\Lup$ components; Condition~\ref{UD:MinPerDom} ensures
  that $H_1(Y;\Z)=0=H_2(Y;\Z)$.
  Compare~\cite[Proposition~2.15]{HolDisk}.)  When $g=0$, the
  three-manifold $Y$ is a three-ball.
\end{rem}

Condition~\ref{UD:TwoBoundariesApieceA} gives a matching $\Matching$
on $\{1,\dots,2n\}$ (a partition into two-element subsets), where
$\{i,j\}\in \Matching$ if $Z_i$ and $Z_j$ can be connected by a path
that does not cross any $\beta_k$.

We will typically abbreviate the data
\[
  \Hup =(\Sigma_0,\{Z_1,\dots,Z_{2n}\},\{\Oup_1,\dots,\Oup_{\Lup}\},
  \{\Xup_1,\dots,\Xup_{\Lup}\}, \]
\[\hskip1in\{\alpha_1,\dots,\alpha_{2n-1}\},\{\alpha^c_1,\dots,\alpha^c_{g+\Lup}\},
  \{\beta_1,\dots,\beta_{g+\Lup+n-1}\}),
\]
and let $\Matching(\Hup)$ be the induced matching on $\{1,\dots,2n\}$.

For upper diagrams, there are two natural generalizations of the notion of
periodic domains and admissibility.

\begin{defn}
  A {\em periodic domain} is a two-chain in
  $\Sigma_0$
  \[ \partial P = \sum_{i} m_i [\alpha^c_i] + \sum_{j} n_j[\beta_j],\]
  with $m_i,n_i\in\Z$.
  A diagram is called {\em admissible} if every non-zero periodic
  domain has positive and negative local multiplicities.
\end{defn}

Viewed as two-chains in $\Sigma$, rather than $\Sigma_0$, the periodic
domains we consider here have vanishing multiplicities at the
punctures corresponding to  $\Zout$, $\zpt$, and $\wpt$.  In
the terminology of~\cite{Bimodules}, these are called {\em provincial
  periodic domains}.

To check admissibility, it helps to have the following observation.
Number the connected components $\{A_1,\dots,A_{\Lup}\}$ as in
Condition~\ref{UD:TwoBoundariesApieceA} so that $A_i$ contains the
markings $\wpt_i$ and $\zpt_i$ specifying $L_i$; number the components
$\{B_1,\dots,B_{\Lup+n-1}\}$ so that their first $\Lup$ components are
also labelled so that $B_i$ contains $\wpt_i$ and $\zpt_i$. Then, it
is easy to see that the periodic domains are spanned by the regions
$A_i-B_i$.

When $\Lup=0$, the corresponding link is empty, and we are considering
the upper diagrams from~\cite{HolKnot}. When $n=0$, we are considering
Heegaard diagrams for links as in Section~\ref{subsec:HeegLinks}.
In particular, when
$n=0$ and $\ell=1$, we can think of these as the double-pointed knot
diagrams from~\cite{HolKnot}.

\begin{defn}
\label{def:HeegaardState}
A {\em Heegaard state} for a marked upper diagram $\Hup$ is a
$d=g+\Lup+n-1$-tuple of points $\x$ in the intersection of the various $\alpha$-and $\beta$-curves,
distributed so that each $\beta$-circle contains exactly one point in $\x$,
each $\alpha$-circle contains exactly one point in $\x$, and no more than 
one point lies on any given $\alpha$-arc.
\end{defn}

For an upper Heegaard diagram with $2n$ outputs, and a Heegaard state $\x$,
we let
\[\alpha(\x)=\{i\in \{1,\dots,2n-1\}\big| \x\cap \alpha_i\neq \emptyset\}.\]

\subsection{Marked middle diagrams}
\label{subsec:MarkedMid}

We will enlarge marked upper diagrams by gluing them to marked middle diagrams.

\begin{defn}
  A {\em marked middle diagram} is the following data:
  \begin{itemize}
  \item a surface $\Sigma_0$ of genus $g$ and $2m$ boundary
    components, which we think of as {\em incoming boundary
      components}, labelled $\Zin_1,\dots,\Zin_{2m}$,
    $2n$ boundary components, which we think of as {\em outgoing boundary
      components}, labelled $\Zout_1,\dots,\Zout_{2n}$
    and $2 \Lmid$ additional
    boundary components labelled
    \[ \Omid_1,\Xmid_1,\dots,\Omid_{\Lmid},\Xmid_{\Lmid}.\]
  \item a collection
    of disjoint, embedded arcs $\{\alphain_i\}_{i=1}^{2m-1}$, so that
    $\alphain_i$ connects $\Zin_i$ to $\Zin_{i+1}$,
  \item a collection
    of disjoint, embedded arcs $\{\alphaout_i\}_{i=1}^{2n-1}$
    (which are also disjoint from the $\alphain_j$), so that
    $\alphaout_i$ connects $\Zout_i$ to $\Zout_{i+1}$
  \item a collection of
    disjoint embedded closed curves $\{\alpha^c_i\}_{i=1}^{g+\Lmid}$
    (which are also disjoint from $\alphain_i$ and $\alphaout_j$)
  \item 
    another collection of embedded, mutually disjoint closed curves
    $\{\beta_i\}_{i=1}^{g+m+n+\Lmid-1}$.  
\end{itemize}
We require this data to also satisfy the following properties:
\begin{enumerate}[label=(MD-\arabic*),ref=(MD-\arabic*)]
\item All the $\alpha$-arcs and curves are transverse to the various $\beta_k$.
\item
Both sets of $\alpha$-and the
$\beta$-circles consist of homologically linearly
independent curves (in $H_1(\Sigma_0)$).
\item
  \label{MD:NoPerDom}
  The span of the homology classes of the
  curves $\{\alpha_i^c\}_{i=1}^{g+\Lmid}$ is linearly independent of the span of
  $\{\beta_i\}_{i=1}^{g+\Lmid+m+n-1}$ in $H_1(\Sigma_0;\Z)$.
\item 
\label{MD:TwoBoundariesApieceA}
The surface obtained by cutting $\Sigma_0$ along
$\beta_1,\dots,\beta_{g+\Lmid+m+n-1}$, which has $m+n+\Lmid$ connected
components $B_1,\dots, B_{m+n+\Lmid}$, is required to contain exactly
two boundary components in each component, so that 
each $\wpt$-boundary is paired with a $\zpt$-boundary.
\setcounter{bean}{\value{enumi}}
\item
  \label{MD:TwoBoundariesApieceB}
  Each component $A_1,\dots,A_{\Lmid+1}$
  of the surface obtained by cutting $\Sigma_0$ along
  $\alpha^c_1,\dots,\alpha^c_{g+\Lmid}$
  is required to contain either all the $Z$-components,  or
  exactly one of the $\wpt$-component and one $\zpt$-component.
  We call the induced one-to-one correspondence between the 
  $\wpt$- and $\zpt$-boundary components the {\em $\alpha$-matching}.
\end{enumerate}
\end{defn}

Sometimes, we abbreviate the data of a middle diagram $\Hmid$.
Moreover, we will write  $\dIn \Hmid=\Zin_1\bigcup \dots\bigcup \Zin_{2m}$
$\dOut \Hmid=\Zout_1\bigcup \dots\bigcup \Zout_{2n}$.

\begin{defn}
  \label{def:BetaMatching}
The $B_1,\dots,B_{m+n+\Lmid}$ induce a matching $\Mmid$ on the boundary
components
\[ \{\Zin_1,\dots,\Zin_{2m},\Zout_1,\dots,\Zout_{2n},
\Omid_1,\dots,\Omid_{\Lmid},\Xmid_1,\dots,\Xmid_{\Lmid}\},\]
called the {\em $\beta$-matching}.
\end{defn}

\begin{defn}
  \label{def:DA-Periodic-Domain}
  A {\em periodic domain} is a relative two-chain for $(\Sigma_0,\Zin)$ with
  \[ \partial P =   \sum_{i} m_i [\alpha^c_i] + \sum_{j} n_j[\beta_j],\]
  with $m_i,n_i\in\Z$,
  satisfying for all $\{i,j\}\in\MatchIn$,
  \[ \weight_{\Zin_i}(P)=\weight_{\Zin_j}(P). \] A diagram is called
  {\em admissible} if every non-zero periodic domain has both positive and
  negative local multiplicities.
\end{defn}

\begin{defn}
  \label{def:Compatible}
  A matching $\MatchIn$ on $\{\Zin_1,\dots,\Zin_{2m}\}$ is said to be {\em
    compatible} with $\Mmid$ if the equivalence relation generated by
  $\MatchIn$ and $\Mmid$ has the property that each equivalence class has
  either exactly one $\wpt$-marking and exactly one $\zpt$-marking, or it
  has exactly two $\Zout$-markings, and this induced matching of the $\wpt$
  and $\zpt$-markings coincides with the one-to-one $\alpha$-matching, as
  defined in Condition~\ref{MD:TwoBoundariesApieceB}.
\end{defn}

Geometrically, we can associate a one-manifold $\Wmid$ with
decorations to a middle diagram. Start from the zero-manifold whose
points correspond to the boundary components of $\Sigma_0$.  To each
component $B_j$ we associate an arc connecting the two boundary
components.  To each component $A_i$ connecting $\wpt_i$ and $\zpt_i$, we
associate an arc connecting $\wpt_i$ to $\zpt_i$, which we call a {\em marked
arc}.  To the incoming matching $\Mdown$,
we associate another one-manifold $\Wup$, consisting of arcs
connecting the pairs of points in the matching. The compatibility
condition can be phrased as follows: each closed component of
the one-manifold $\Wmid\cup\Wup$ contains exactly one marked arc.

We will need the following notion of a compatible orientation on
$\Wmid\cup\Wup$ to define the bimodule associated to a middle diagram
and an incoming matching.

\begin{defn}
  \label{def:OneManifoldWithMarkings}
  A {\em one-manifold with markings} is a one-manifold
  $W$, equipped with a set of embedded arcs, each of whose endpoints are
  labelled with an $\zpt$ and an $\wpt$. A {\em compatible orientation on the
    one-manifold with markings} is an orientation of $W$ whose
  restriction to each marked arc orients it as a path from the
  $\wpt$-marking to the $\zpt$-marking. We will use these compatible
  orientations to define the bimodule associated to a middle diagram.
\end{defn}

The definition of Heegaard states (Definition~\ref{def:HeegaardState})
generalizes immediately to the case of middle diagrams, with the
understanding that now the number of $\beta$-circles is given by
$d=g+\Lmid+m+n-1$.

For a Heegaard state $\x$ in a middle diagram, let
\begin{align*}
  \alphain(\x)&=\{i\in \{1,\dots,2m-1\}\big| \alphain_i\cap \x\neq
  \emptyset \}\\ \alphaout(\x)&=\{i\in \{1,\dots,2n-1\}\big|
  \alphaout_i\cap \x\neq \emptyset\}.
\end{align*}

\begin{defn}
  \label{def:ExtendedMiddle}
An {\em extended marked middle diagram} is a middle diagram,
together with two new boundary components, $\Zmid_0$ and $\Zmid_1$, four new
$\alpha$-arcs, $\alphain_{0}$, $\alphain_{2m}$,  $\alphaout_0$,
and $\alphaout_{2n}$, and two new 
$\beta$-circles $\beta_0$ and $\beta_0'$,
arranged so that
\begin{itemize}
  \item $\alphain_{0}$ connects
    $\Zin_{1}$ to $\Zmid_0$,
  \item $\alphaout_{0}$ connects $\Zout_{1}$ to
    $\Zmid_0$,
  \item $\beta_0$ separates $\Zmid_0$ from all the other
    boundary components of $\Sigma_0$,
  \item $\alphain_{2n}$ connects $\Zin_{2m}$ to $\Zmid_1$
  \item $\alphaout_{2m}$ connects $\Zout_{2n}$ to $\Zmid_1$
  \item $\beta_0'$ separates $\Zmid_1$ from all the other boundary components
    of $\Sigma_0$.
  \item
    if we fill in $\Zmid_0$ and $\Zmid_1$ with disks,
    and delete the newly introduced arcs and circles,
    the result is a middle diagram.
\end{itemize}
\end{defn}

For an extended middle diagrams, 
$\alphain(\x)\subset \{0,\dots,2m\}$ and
$\alphaout(\x)\subset \{0,\dots,2n\}$.

(Extended marked middle diagrams are the natural adaptation
of the notion of extended middle diagrams as defined in~\cite{HolDisk}.)

\subsection{Marked lower diagrams}
\label{subsec:MarkedLower}

\begin{defn}
  A {\em marked lower diagram} is a marked middle diagram with no
  outgoing boundary components (or arcs), and one additional $\beta$
  circle, i.e. the $\beta$ circles are numbered
  $\{\beta_i\}_{i=1}^{g+n+\Ldown}$, with $n>0$, together with one distinguished
  $\wpt$-$\zpt$-pair, $\Odown_1$, $\Xdown_1$, so that
  Conditions~\ref{MD:TwoBoundariesApieceA}
  and~\ref{MD:TwoBoundariesApieceB} are replaced by the following:
  \begin{enumerate}[label=(LD-\arabic*),ref=(LD-\arabic*)]
      \setcounter{enumi}{\value{bean}}
  \item The surface obtained by cutting $\Sigma_0$ along
    $\beta_1,\dots,\beta_{g+\Ldown+n}$, has $n+1$ components
    $B_1,\dots, B_{n+\Ldown+1}$, each of which contains exactly two boundary components,
    which can be of the following types:
    \begin{itemize}
      \item $B_i$ contains $\wpt_j$ and $\zpt_j$ with $j\neq 1$
      \item $B_i$ contains $\wpt_1$ and a component of type $Z$
      \item $B_i$ contains $\zpt_1$ and a component of type $Z$
      \item $B_i$ contains two boundary components of type $Z$.
        \end{itemize}
  \item Each component $A_1,\dots,A_{\Ldown}$ of the surface obtained
    by cutting $\Sigma_0$ along $\alpha^c_1,\dots,\alpha^c_{g+\Ldown}$
    is required to contain either all the $Z$ components and $\Odown_1$ and $\Xdown_1$;
    or exactly one $\Odown_j$ and exactly one $\Xdown_j$ (for $j\neq 1$).
    \end{enumerate}
\end{defn}

For a lower diagram with boundary
$\Zin_1,\dots,\Zin_{2m},\Oin_1,\Xin_1,\dots,\Oin_{\Lin},\Xin_{\Lin}$,
we can define the associated matching $\Matching$ (exactly as in
$\Mmid$ for middle diagrams).  Similarly, given a matching $\MatchIn$
on $\{\Zin_1,\dots,\Zin_{2m}\}$, we can define a compatibility
condition with $\Matching$ as in Definition~\ref{def:Compatible}.
Letting $W$ be the marked one-manifold associated to the lower
diagram, we can form the marked one-manifold $W\cup\Wup$.  For lower
diagrams, the orientation on $W\cup\Wup$ is uniquely specified by the
compatibility condition of
Definition~\ref{def:OneManifoldWithMarkings}.

\subsection{Gluing diagrams}

If $\Hup$ is an upper diagram, and $\Hmid$ is a middle diagram, and an
identification $\partial\Hup=\dIn\Hmid$, we can form a new upper
diagram $\Hup\cup_{\Zin} \Hmid$. 
For example, if $\{\Oup_1,\dots,\Oup_{\Lup}\}$ and 
$\{\Omid_1,\dots,\Omid_{\Lmid}\}$
resp.
$\{\Xup_1,\dots,\Xup_{\Lup}\}$ and 
$\{\Xmid_1,\dots,\Xmid_{\Lmid}\}$
denote the $\wpt$-marked resp. $\zpt$-marked
boundaries of $\Hup$ and $\Hmid$, then 
the $\wpt$- and $\zpt$-marked boundaries of 
$\Hup\cup_{\Zin} \Hmid$ are
\[ \{\Oup_1,\dots,\Oup_{\Lup},\Omid_1,\dots,\Omid_{\Lmid}\}
\qquad{\text{and}}\qquad
\{\Xup_1,\dots,\Xup_{\Lup},\Xmid_1,\dots,\Xmid_{\Lmid}\}.\]

If $\Hup$ and $\Hmid$ are admissible,
then the glued diagram is admissible, as well.

\subsection{Heegaard diagrams from link projections}
\label{subsec:HeegLinkProj}

A knot projection has a naturally associated doubly-pointed Heegaard
diagram, as described in~\cite{AltKnots}, studied further
in~\cite{HolKnot}. We must modify this algorithm to take into account
marked edges (which are treated slightly differently from the
distinguished edge). The construction is summarized in
Figure~\ref{fig:LinkHeeg}.

\begin{figure}[h]
 \centering
 \input{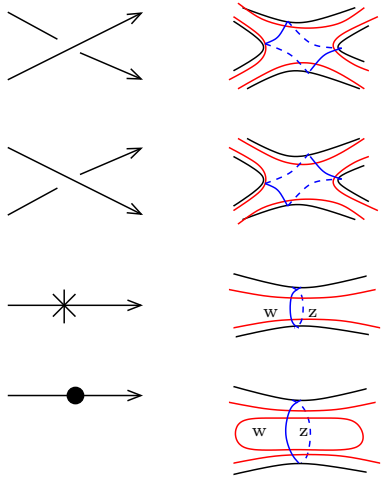}
 \caption{{\bf Heegaard diagram for a marked minimum.}}
 \label{fig:LinkHeeg}
 \end{figure}

The resulting diagram is called the {\em small Heegaard diagram associated
  to a marked link projection}.

\begin{prop}
  \label{prop:KauffLink}
  There is a $2^{\ell-1}$-to-one correspondence between Heegaard
  states in the above diagram with the Kauffman states
  of~\cite{Kauffman}. The Maslov and Alexander gradings are given by
  multiplying the monomial local contributions as in
  Figure~\ref{fig:KauffLink}.
\end{prop}

The above proposition follows quickly from the methods from~\cite{AltKnots}.
We will see it as an easy consequence of
Proposition~\ref{prop:KauffLink2} below.

There is a basis for the periodic domains for the link projection
which correspond to the non-distinguished components of the link:
their local multiplicity is $1$ in the portion of the Heegaard diagram
around the given link component. Thus, the diagram constructed above is
not admissible. To construct an admissible diagram, we will need some
additional data.

\begin{defn}
  \label{def:Relevant}
  A marked link diagram is called {\em relevant} if the following two properties
hold:
\begin{itemize}
  \item the marking on each component occurs
    at the minimum of the height function restricted to the link components
  \item the distinguished marking occurs at the global minimum of the height function.
  \item each minimum is leftmost; i.e. the intersection of the
    horizontal line through each minimum meets the rest of the link to
    the right of the minimum.
\end{itemize}
\end{defn}

The first two properties are  the canonical marking of the diagram discussed
in the introduction; the second property implies that 
the markings are adjacent to the infinite (marked) region.

Fix a relevant link diagram, connect the marking on some component to
the adjacent strand by a dotted arc.  We then isotope each such dotted
arc which is otherwise disjoint from the knot projection. We shrink
this dotted arc, to form an {\em admissibility marker}.
Crossings in the diagram and admissibility markers together can be thought of as {\em generalized crossings}.
In particular, at each admissibility marker, there are also
four adjacent regions, which we think of as $N$, $S$, $W$, and
$E$. We mark these so that the $S$  region is part of the infinite region, and the
short dotted arc runs $W$ to $E$, as shown in
Figure~\ref{fig:AdmissMarker}.

\begin{figure}[h]
 \centering
 \input{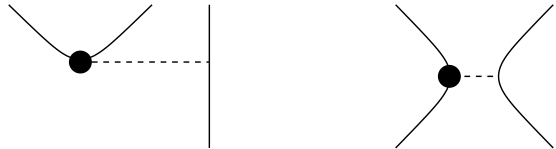}
 \caption{{\bf Admissibility markers.}  A portion of the canonically
   marked diagram on the left is decorated by a dotted arc as shown,
   and isotoped to the picture on the right.}
 \label{fig:AdmissMarker}
 \end{figure}

The regions in the diagram will be thought of as the connected
components of the complement of the diagram, together with the
$\ell-1$ dotted arcs at each marking. Two of these regions are distinguished by being
adjacent to the distinguished edge.

We extend the notion of Kauffman states for relevant diagrams, as
follows.  Once again, a Kauffman state associates one of the four
adjacent regions at each crossing.  Moreover, our extended Kauffman
states also can associate one of the two adjacent regions at each
admissibility marker: these regions can be $N$ or $W$ (but not $E$ or
$S$).  Our Kauffman state is constrained by the requirement that there
is a one-to-one correspondence between generalized crossings and the
indistinguished regions in the diagram.

\begin{figure}[h]
 \centering
 \input{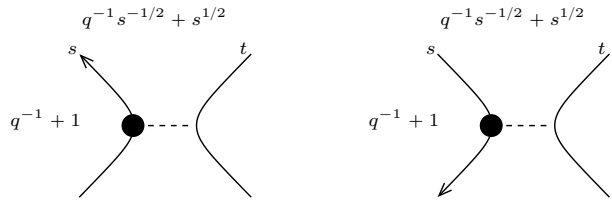}
 \caption{{\bf Local Kauffman states at an admissibility arc.} 
   There are two local Kauffman states ($N$ and $W$) that can appear
   in a generalized Kauffman state near an admissibility marking.
   These contribute the displayed monomial in $q$ and $s$..}
 \label{fig:KauffAdmissMarker}
 \end{figure}

To construct the {\em Heegaard diagram associated to the relevant link
  projection}, we associate the piece of diagram from
Figure~\ref{fig:AdmissHeeg} to each admissibility marker (and the
pieces associated to crossings and specially marked edge as in
Figure~\ref{fig:LinkHeeg}).

\begin{figure}[h]
 \centering
 \input{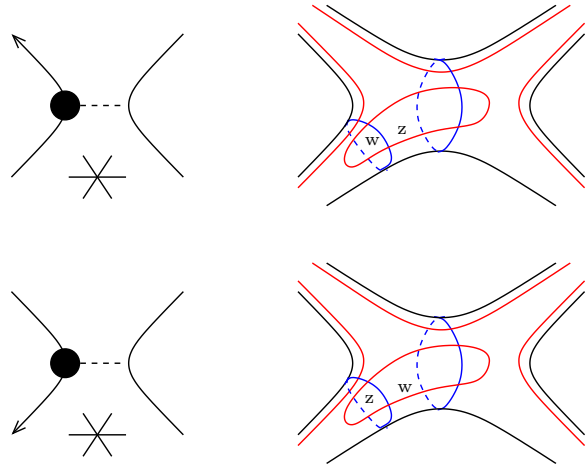}
 \caption{{\bf Heegaard diagram for an admissibility arc.}  
   The admissibility marker at the left is replaced by the piece
   of Heegaard diagram on the right. (The star is drawn to remind
   the reader that the Southern quadrant is adjacent to the distinguished edge.)
 }
 \label{fig:AdmissHeeg}
\end{figure}

For a relevant link diagram, we will identify the Kauffman states with
the Heegaard states of the associated link diagram
(c.f. Proposition~\ref{prop:KauffLink2} below), in a way that respects
the the Maslov and Alexander gradings. Our verification use methods of
Kauffman~\cite{Kauffman}, suitably adapted.

\begin{figure}[h]
 \centering
 \input{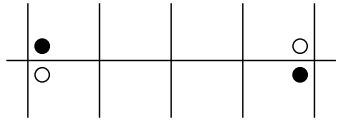}
 \caption{{\bf Schematic of two Kauffman states (black and white)
     connected by a clock move.}  
   The two Kauffman states are required to agree at all other crossings;
   the edges in the horizontal line are required not adjacent to either
   distinguished region.
 }
 \label{fig:ClockMove}
\end{figure}

Specifically, recall that two Kauffman states $\x$ and $\y$ are
connected by a {\em clock move} if there are two crossings where $\x$
and $\y$ are different and the same everywhere else, and there is a
path of edges connecting those two crossings which is never adjacent
to one of the two distinguished regions, see for example 
Figure~\ref{fig:ClockMove}.  According to Kauffman~\cite{Kauffman}, any
two Kauffman states for a connected projection can be connected by a
sequence of clock transformations.  We need the following variant of
this fact:

\begin{defn}
  Let $P$ be a knot diagram with a distinguished crossing $x$, whose
  four quadrants are labelled $N,S,E,W$.  The {\em $NS$ resolution of
    $P$ at $x$} is the diagram obtained from $P$ by resolving the
  crossing so that the $N$ and $S$ become connected to one another.
  The {\em $EW$ resolution} is the other resolution.
\end{defn}

\begin{defn}
  Let $C$ be a set of crossings of an oriented marked planar diagram,
  and mark the quadrants of the projection at each $c\in C$ as $N$,
  $W$, $E$, $S$.  We say that a Kauffman state is of type $C$-NW, if
  at each $c\in C$, the Kauffman state is in the $N$ quadrant or the
  $W$-quadrant.
\end{defn}

\begin{defn}
  A set $S$ of Kauffman states is called {\em clock connected} if any two
  elements of $S$ can be connected by a sequence of clock transformations,
  each of which connects two elements of $S$.
\end{defn}

In this language, Kauffman's clock theorem states that for a connected
projection, the set of all Kauffman states is non-empty and
clock-connected.

\begin{prop}
  \label{prop:ClockConnected}
  Fix an oriented planar diagram $P$ Let $P$ be an oriented marked
  planar diagram, and choose a collection $C$ of crossings in $P$,
  which are all adjacent to one region $R$, which is one of the two
  distinguished regions in $P$. Label the quadrants at each $c\in C$
  so that the quadrant contained in $R$ is labelled with an $S$.  
  Suppose that the diagram $P'$ obtained as the $NS$ resolution of $P$
  at each $c\in C$ is connected.
  Then, the set of Kauffman states of type $C$-NW is clock-connected.
\end{prop}

\begin{proof}
  We prove this by induction on the number of elements in $C$.
  If $C$ is empty, then the theorem follows from Kauffman's Clock theorem
  applied to $P$.

  The Kauffman states that are of type $N$ at $c\in C$ correspond to
  Kauffman states in the $NS$ resolution of $P_c$ of $P$ at $c$, and
  the clock transformations between such Kauffman states correspond to
  clock transformations in $P_c$.  Thus, it follows by the inductive
  hypothesis, the set of Kauffman states of type $C$-$NW$ for which at
  least one crossing in $C$ is of type $N$ is clock-connected.

  Consider any Kauffman state of type $C$-$W$. Since the Kauffman
  states for $P$ are clock-connected (Kauffman's theorem again), for
  any Kauffman state $x_1$ of type $C$-$W$, we can find a sequence
  $x_1,\dots,x_m$ of Kauffman states so that $x_i$ is connected to
  $x_{i+1}$ by a clock move, and $x_m$ is not of type
  $C$-$W$. Consider the smallest $i$ for which $x_i$ is not of type
  $C$-$W$. This means that there is one $c\in C$ so that $x_i$ is not
  of type $W$ at $c$, but $x_{i-1}$ is of type $W$ at $c$. Since
  the $S$ quadrant is distinguished, It follows
  that $x_i$ must be of type $N$ at $c$. 
\end{proof}

For a relevant diagram, we consider Kauffman states, now connected generalized Kauffman states.
In particular, at the admissibility markers, these states are required to occupy the $N$ and the $W$ quadrants only.

\begin{cor}
  For a relevant link diagram which is connected in the complement of
  the edge markers, the set of generalized Kauffman states is
  clock-connected.
\end{cor}

\begin{proof}
  This follows from Proposition~\ref{prop:ClockConnected} (thinking of
  the generalized Kauffman states as Kauffman states for a connected
  diagram, which are constrained to occupy positions $N$ and $W$).
\end{proof}

\begin{prop}
  \label{prop:KauffLink2}
  The Heegaard diagram associated to a relevant link diagram is an
  admissible Heegaard diagram, and its Heegaard states correspond to
  Kauffman states for the admissibly marked diagram.  Moreover, the
  Alexander and Maslov gradings are as specified in 
  Figures~\ref{fig:KauffLink} and~\ref{fig:KauffAdmissMarker}.
\end{prop}

\begin{proof}
For admissibility we argue as follows. There is an ordering on link
components, $L_1,\dots,L_{\ell}$, in increasing order of the value of
the global minimum; for example, $L_\ell$ contains the special
minimum.  The space of periodic domains is spanned by 
domains $P_i$ supported in a
neighborhood of each $L_i$ for $i=2,\dots,\ell$. Find the largest $k$
so that $P_k$ appears in a given expansion. If it occurs with positive
local multiplicity, then we can find a small bigon in the Heegaard
diagram for the corresponding admissibility marker (i.e. the bigon
from $y$ to $x$ locally in Figure~\ref{fig:HSAdmiss}) where the
periodic domains has has negative local multiplicity. This verifies admissibility.

  Label the Heegaard states locally near each admissibility arc
  as in Figure~\ref{fig:HSAdmiss}.
\begin{figure}[h]
 \centering
 \input{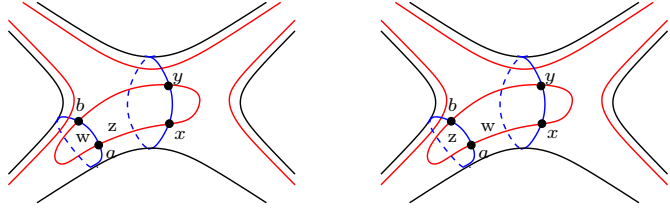}
 \caption{{\bf Local Heegaard states near the
     admissibility markers.}}
 \label{fig:HSAdmiss}
 \end{figure}
 
 This gives a correspondence between Heegaard states and Kauffman states:
 $a$ and $b$ locally correspond to the two monomials
 in the northern quadrants of Figure~\ref{fig:KauffAdmissMarker};
 $x$ and $y$ locally correspond to the two terms in
 the western quadrant in Figure~\ref{fig:KauffAdmissMarker}.

 We check next that the local contributions of these Kauffman states
 agree with the Alexander and Maslov gradings, up to overall additive
 constants.

 Suppose that the marked strand is oriented upwards; i.e. that we are
 considering the Heegaard diagram on the left in
 Figure~\ref{fig:HSAdmiss}.  First we compare the contributions of
 states associated to the same quadrant, using domains that are
 locally contained in the marked region; i.e. by inspecting
 Figure~\ref{fig:HSAdmiss}. For example, there is a bigon in
 $\pi_2(a,b)$ (i.e with Maslov grading $1$) that crosses $\zpt$ but not
 $\wpt$. It follows that for two
 states occupying the same corners,
 \[\begin{array}{ll}
   \Alex(b)=\Alex(a)+[\mu_s] &
   \Mgr(b)=\Mgr(a)+1 \\
   \Alex(x)=\Alex(y) &
   \Mgr(y)=\Mgr(x)+1, \\
 \end{array}\]
 where here $[\mu_s]\in H_1(S^3\setminus L)$ is the generator which is the meridian for the $s^{th}$ component
 of $L$.

 Next, as in the proof of~\cite[Theorem~1.2]{ClassKnots}, we check
 that the Kauffman contributions transform as expected under
 clock moves. This is done in
 Figure~\ref{fig:CheckLocalContribs}.

 A similar computation can be done when the strand is oriented
 oppositely, to complete the identification of the Maslov and
 Alexander gradings, up to an overall additive indeterminacy.

 \begin{figure}[h]
   \centering
   \input{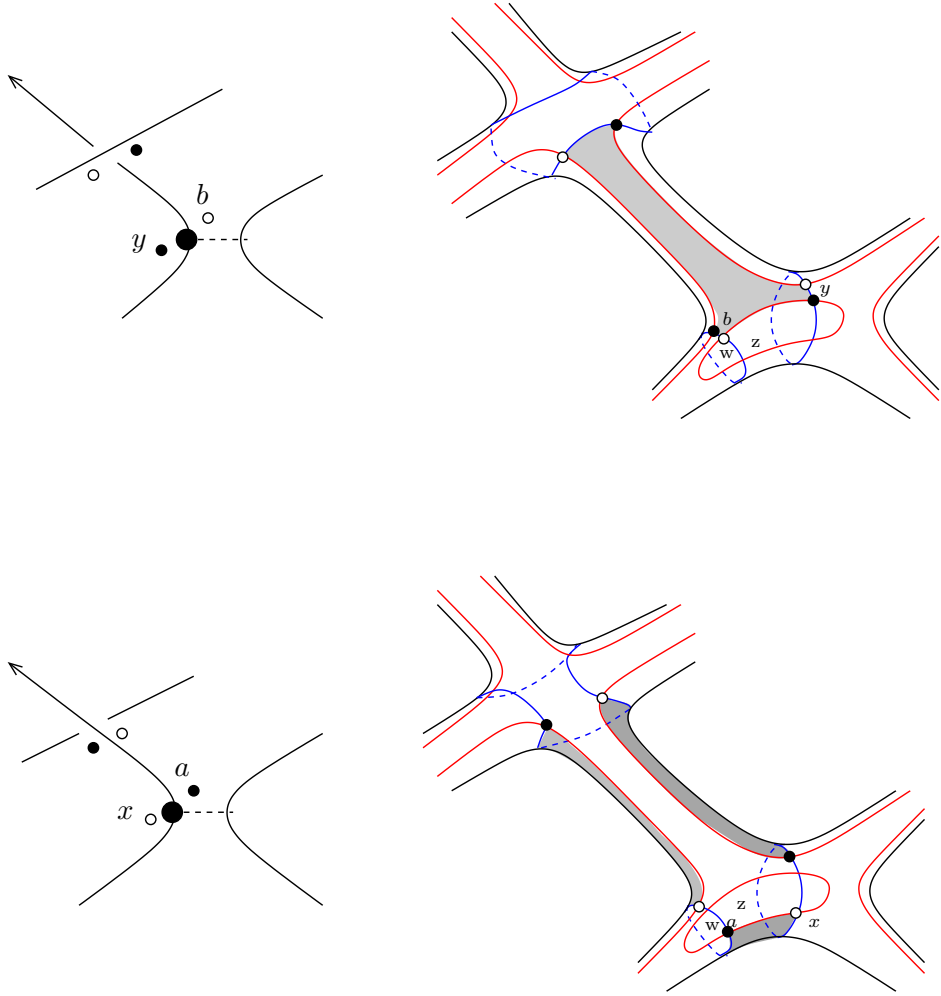}
   \caption{{\bf Checking local contributions.}}
   \label{fig:CheckLocalContribs}
 \end{figure}

 To remove the additive indeterminacy of the Maslov grading, we argue
 as in~\cite[Lemma~2.5]{ClassKnots}. After a sequence of handleslides
 back over $\zpt$-marked regions, we obtain a new diagram, where the
 $\beta$-curve at a each crossing is replaced by meridian, which meets
 $\alpha$-curves belonging to two (rather than possibly) four regions.
 Specifically, label the four edges at each crossing. The new
 $\beta$-curve, then, will be a meridian for the second edge we
 encounter as we traverse the knot. Once again, there is a
 correspondence between Heegaard states and Kauffman states for the
 resulting diagram, with the understanding that Kauffman states for
 the new diagram are allowed to be adjacent only to this second
 edge. In~\cite{ClassKnots}, we proved that there was one such
 state. (In that paper, we were considering the ``penultimate'', which
 is the third edge we encounter, instead of the second; but this
 discrepancy is due to our different orientation conventions on the
 knot.)  That generator in turn correspond to $2^{\ell-1}$ Heegaard
 states, according to whether we are using locally the intersection
 point $a$ or $b$; and all of these live on in homology. The absolute
 gradings are verified by noting that maximal Maslov grading is $0$.
 Since handleslides preserve the Maslov grading of the generator, the
 formula follows.

 \begin{figure}[h]
   \centering
   \input{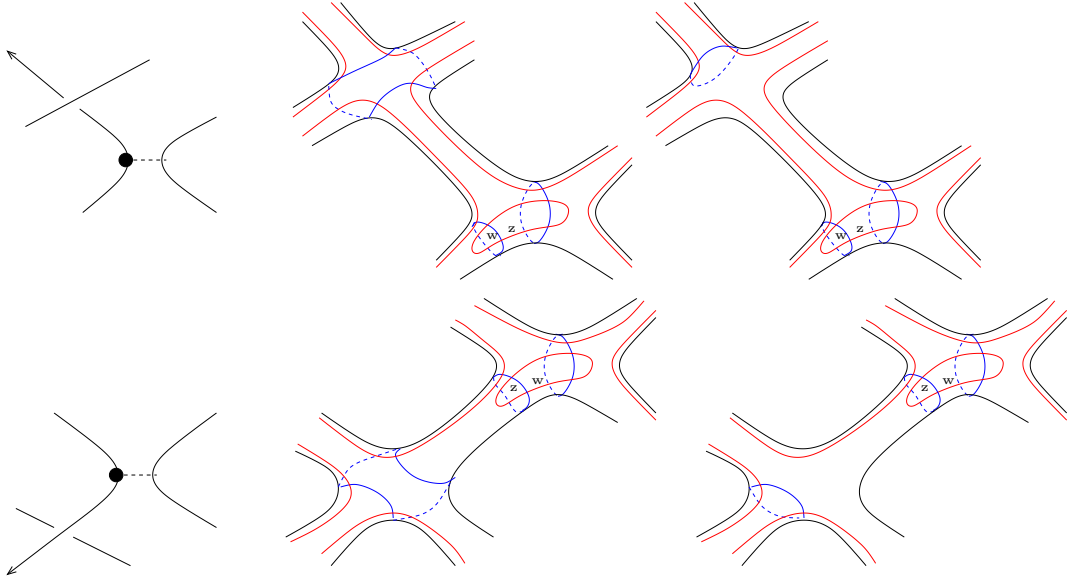}
   \caption{{\bf Kauffman state with Maslov grading zero.}
   Handleslide locally over the $\zpt$ marking to obtain the Heegaard diagram
   on the right.}
   \label{fig:HandleslideBack}
 \end{figure}

 The symmetry from Equation~\eqref{eq:Symmetry} holds. This can be
 seen on the level of local Kauffman contributions
 (Figures~\ref{fig:KauffLink} and~\ref{fig:KauffAdmissMarker}).
 Specifically, if $s$ is the variable associated to a given strand,
 then substituting $s\goesto (q^2 s)^{-1}$ to the monomial
 contribution gives $q^{\pm 1/2}$ times the monomial contribution
 obtained by reversing the orientation of the strand, where $\pm $ is simply the sign of the crossing. Thus,
 for a given Kauffman state $\x$, if
 \[ \Alex_{\oL}(\x)=(s_1,\dots,s_\ell) \qquad{and}\qquad\Mas_{\oL}(\x)=d \]
 then
 \[ \Alex_{\oL'}(\x)=(s_1,\dots,s_{i-1},-s_i,s_{i+1},\dots,s_\ell)
 \qquad{and}\qquad
 \Mas_{\oL}(\x)-2s_i=\Mas_{\oL'}(\x)-\kappa_i.\]
 Since the chain complexes $\CFLa(\orL)$ and $\CFLa(\orL')$ are isomorphic,
 we have verified Equation~\eqref{eq:Symmetry}.
 The absolute Alexander gradings follow.
\end{proof}

These diagrams can be sliced: a generic horizontal slice in a marked
link projection corresponds to a collection of curves in the Heegaard
surface that decomposes the Heegaard surface into a marked upper
diagram and a marked lower diagram.


\begin{figure}[h]
 \centering
 \input{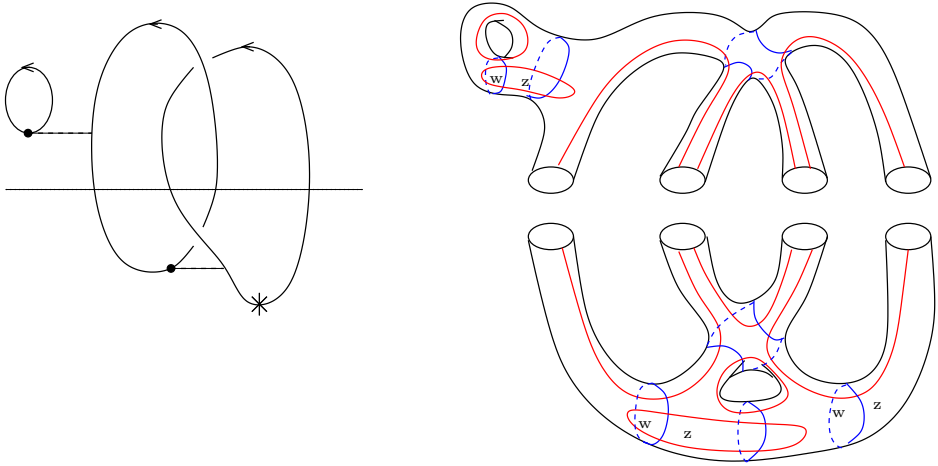}
 \caption{{\bf Slicing the Heegaard diagram of a marked link
     projection.}  Start from the marked link projection on the left,
   sliced in half along the dotted line, to obtain the upper and lower
   marked diagrams on the right.}
 \label{fig:SliceDiagram}
\end{figure}

\section{Algebra}
\label{sec:Algebra}

\subsection{The bordered algebras}
\label{subsec:OurAlgebras}

Recall the bordered algebra $\Blg(m,k)$ constructed
in~\cite[Section~3.2]{BorderedKnots}.
As in~\cite{HolKnot}, we specialize to the case $\Blg(n)=\Blg(2n,n)$,
equipped with its idempotent ring $\IdempRing(n)=\IdempRing(2n,n)$.
We will consider various subrings of
its idempotent ring $\IdempRing(n)$, as follows:

\begin{defn}
  \label{def:IdempSubring}
  Let $\IdempRing_{* \geq i}$ denote the subalgebra of $\IdempRing$
  generated by idempotent states $\x=x_1<\dots<x_k$ with $x_1\geq i$;
  similarly, let $\IdempRing_{*\leq i}$ denote the subalgebra of
  $\IdempRing$ generated by $\x$ so that $x_k\leq i$. Let
  $\IdempRing_{i\leq *\leq j}=\IdempRing_{*\geq i}\cdot \IdempRing_{*\leq  j}$.
\end{defn}

The algebra $\Blg(n)$ has a subalgebra $\IdempRing_{1\leq *\leq
  2n-1}\cdot \Blg(n)\cdot \IdempRing_{1\leq *\leq 2n-1}$, denoted
$\Clg(n)$ in~\cite{HolKnot}.  This is naturally the ring over which
the type $D$ structures of an upper diagram are defined. More
precisely, upper diagrams come equipped with a matching $\Matching$ on
$\{1,\dots,2n\}$. The matching specifies a curvature
\[ 
\mu_0^{\Matching}=\sum_{\{i,j\}\in\Matching} U_i U_j.
\]

Given an upper Heegaard diagram $\Hup$ with $2n$ $Z$-boundary components
and $\ell$ $\wpt$-boundaries (and $\ell$ $\zpt$-boundaries),
we can associate a matching $\Matching$ on $\{1,\dots,2n\}$.
Let $\cClg$ be the algebra
$\cClg(n)\otimes \Field[w_1,z_1,\dots,w_\ell,z_\ell]$ equipped
with curvature $\mu_0^{\Matching}$. We will show in Section~\ref{sec:DAmods}
how to construct an associated curved type $D$ structure
$\lsup{\cClg}\Dmod(\Hup)$. (In the case where $\ell=0$, this is the construction from~\cite{HolKnot}.)

Recall that $\Blg$ has a grading, called the {\em $\Delta$-grading},
which is $-1$ times the sum of all the weights of the algebra
elements; e.g. $\Delta(L_i)=-1/2$, $\Delta(U_i)=-1$.

\subsection{Simplifying type $D$ structures}

Let 
\[ \BigBlg=\Blg(n)\qquad{\text{or}}\qquad\Blg(n)\otimes \Field[w_1,z_1,\dots,w_\ell,z_\ell] \]
for some value of $n$ (and $\ell$). The algebra $\BigBlg$ is filtered
by weight, in the following sense. There is an ideal $\BigBlg_+$
generated by algebra elements with positive weight (either in $\Blg$
or in the $w_i$ or $z_i$) with quotient $\BigBlg/\BigBlg_+\cong
\IdempRing_{n}$.

Suppose that $\lsup{\cBigBlg}X$ is a (curved) type $D$ structure. 

\begin{lemma}
  \label{lem:SimplifyTypeD}
  A finitely generated type $D$ structure
  $\lsup{\cBigBlg}X$ is homotopy equivalent to a different type $D$ structure
  $\lsup{\cBlg}X'$ with the property that
  \[ \delta^1\colon X'\to \BigBlg_+\otimes X'. \] 
\end{lemma}

\begin{proof}
  This is a familiar argument in homological algebra: if
  $\delta^1(X')\not\in\BigBlg_+\otimes X'$, there are $p,q\in X'$ so
  that $\delta^1(p)\in q+\BigBlg_+\otimes X'$.  Define a new homotopy
  equivalent type $D$ structure with two fewer generators obtained by
  contracting the arrow from $p$ to $q$.
  The result now follows by induction on the rank of $X$ (as a left
  $\IdempRing$-module).
\end{proof}

\subsection{The algebra $\nDuAlg$}

In~\cite{HolKnot}, we introduced also algebras $\nDuAlg(n,\Matching)$
where $\Matching$ is a matching on $\{1,\dots,2n\}$. The algebra is
obtained by adding adding elements $\{E_i\}_{i=1}^{2n}$ to
$\Blg(2n,n+1)$. These elements satisfy the relations $E_i^2=0$,
\[ [E_i,E_j] =\left\{\begin{array}{ll}
1 &{\text{if $\{i,j\}\in\Matching$}} \\
0 &{\text{otherwise}}
\end{array}\right.\]
The algebra is equipped with a differential satisfying $d E_i=U_i$.

In Section~\ref{HK:subsec:nDuAlg}, we  introduced a 
$DD$ bimodule
$\lsup{\cBlg,\nDuAlg}\CanonDD$, as follows.
Fix $n\geq 1$ and a matching $\Matching$ on $\{1,\dots,2n\}$,
and let $\nDuAlg=\nDuAlg(n,\Matching)$, $\cBlg=\cBlg(n,\Matching)$.
Generators of $\lsup{\cBlg,\nDuAlg}\CanonDD$, as a vector space,
correspond to $n$-element subsets $\x\subset\{0,\dots,2n\}$, i.e.
$I$-states for $\Blg(2n,n)$.
Let $\gen_\x$ be the generator corresponding to $\x$. The  left $\IdempRing(2n,n)\otimes \IdempRing(2n,n+1)$-module 
structure is specified by
\[ (\Idemp{\x}\otimes \Idemp{\{0,\dots,2n\}\setminus \x}) \cdot \gen_\x = \gen_\x.\]
The differential is specified by the element
\[
A = \sum_{i=1}^{2n} \left(L_i\otimes R_i + R_i\otimes L_i\right) + \sum_{i=1}^{2n}
U_i\otimes E_i \in
  \Blg\otimes \nDuAlg,\]
  \[ \delta^1 \colon \CanonDD \to \Blg\otimes\nDuAlg \otimes
  \CanonDD.\] by $\delta^1(v)=A\otimes v$.

\subsection{Working over $\cClg$}

In this paper, we will consider type $D$ structures over $\cBlg$
which are actually induced by type $D$ structures over $\cClg$.
We write such type $D$ structures $\lsup{\cBlg}\iota_{\cClg}\DT \lsup{\cClg}Y$.
Concretely, a type $D$ structure $\lsup{\cBlg}Y$ is of this type if 
$\Idemp{1\leq * \leq 2n-1}\cdot Y = Y$.

The first simple observation we will use about such type $D$
structures is the following:

\begin{lemma}
  \label{lem:Injection}
  If 
  $\lsup{\cBlg}\iota_{\cClg}\DT~ \lsup{\cClg}Y\simeq
  ~\lsup{\cBlg}\iota_{\cClg}\DT~ \lsup{\cClg}Y'$,
  then $\lsup{\cClg}Y\simeq \lsup{\cClg}Y'$.
\end{lemma}
\begin{proof}
  Our morphisms (and homotopies) of type $D$ structures are, by definition, left
  $\IdempRing$-equivariant.  The result follows easily.
\end{proof}

\begin{lemma}
If 
\[ \lsup{\cBlg}X' \simeq 
\lsup{\cBlg}\iota_{\cClg}\DT~ \lsup{\cClg}Y, \]
and $\delta^1\colon X' \to \cBlg_+\otimes X'$,
then in fact $\delta^1$ takes $X$ to $\cClg_+\otimes X'$.
\end{lemma}

\begin{proof}
  Suppose that $X'$ has some element $p$ with the property that
  $\Idemp{\x}\cdot p = p$ with $0$ or $2n \in \x$.
  Then, the homotopy equivalence implies that
  \[ p=\delta^1\circ h + h \circ \delta^1;\] and our hypothesis on
  $\delta^1$ ensures that the right hand side is in $\Blg_+\otimes
  X$. This contradicts the fact that $p\in \IdempRing(n)\otimes X$.
\end{proof}

In view of the above lemma, if $X$ is any type $D$ structure
\[ \lsup{\cBlg}X \simeq \lsup{\cBlg}\iota_{\cClg}\DT~
\lsup{\cClg}Y, \] we can find an explicit form for $\lsup{\cClg}Y$ by
applying the algorithm of Lemma~\ref{lem:SimplifyTypeD}.

\subsection{Boundedness and tensor products}

Our type $DA$ bimodules over $\Blg(n)\otimes
\Field[w_1,z_1,\dots,w_\ell,z_\ell]$ will typically be bounded above
in Maslov grading; but not below. After all, each variable $w_i$ and
$z_i$ drops Maslov grading by one.

\begin{defn} 
  Let $\lsup{\Alg}X_{\Blg}$ be a DA bimodule over algebras
  $\Alg$ and $\Blg$.
  A type $DA$ bimodule is called {\em unital} if 
  \[ \delta^1_{k+1}(\x,a_1,\dots,a_k)
  =\left\{\begin{array}{ll}
        \x &{\text{if $k=1$ and $a_1=1$}} \\
        0 &{\text{if $k>1$ and some $a_i=1$}}.
        \end{array}\right.\]
  A type $DA$ bimodule is called {\em bounded}
  if for every integer $M$ (which one can think of as a negative integer
  with large absolute value) there is an $N$ so that
  $\Delta(\delta^1_{k+1}(\x,a_1,\dots,a_k))\geq M$
  we can conclude that
  $\sum_{i=1}^k \Delta(a_i)\geq N$.
\end{defn}

Our type $DA$ bimodules will be both bounded and unital.

\begin{prop}
  Let $\cBigBlg_i$ be algebras of the form $\Blg(n)$ or
  $\Blg(n)\otimes\Field[w_1,z_1,\dots,w_\ell,z_\ell]$ for $i=1,2$.  If
  $\lsup{\cBigBlg_2}X_{\cBigBlg_1}$ is a curved type $DA$ bimodule
  which is strictly unital, bounded, and graded. Let
  $\lsup{\cBigBlg_1}Y$ be a finitely generated type $D$ structure
  which is graded.
  Then, the sums defining type type $D$ structure
  on $X\DT Y$ are finite.
\end{prop}

\begin{proof}
  Gradings give a lower bound on $\delta(\x\otimes \y)$, and hence, by
  the boundedness of $X$, a lower bound on the grading of $a_1\otimes
  \dots\otimes a_k$. Unitality then gives the desired upper bound on
  $k$.
\end{proof}

\newcommand\OmegaInEv{\widecheck{\mathbf\Omega}^+}
\newcommand\OmegaInOdd{\widecheck{\mathbf\Omega}^-}
\newcommand\orW{\vec{W}}
\newcommand\pin{\widehat{p}}
\newcommand\pout{\widecheck{p}}
\newcommand\rhosin{\rhos^{\wedge}}
\newcommand\rhosout{\rhos^{\vee}}
\newcommand\SourceMid{\Source}
\newcommand\IdempIn{\widecheck{I}}
\newcommand\IdempOut{\widehat{I}}

\newcommand\Brs{{\mathcal B}_{\{r,s\}}}
\section{Holomorphically defined modules}
\label{sec:DAmods}

We sketch here the fairly straightforward modifications
needed to adapt the modules from~\cite{HolKnot} to the case of links.

\subsection{Modules for upper diagrams}

Let $\Hup$ be an admissible marked upper diagram with $2n$ boundary
circles $\Zout_1,\dots,\Zout_{2n}$. Let $\Mup$ be the induced
matching. Our aim here is to define a curved type $D$ structure over the algebra
\[ \Clgout=\Clg(n)\otimes \Field[w_1,\dots,w_\ell,z_1,\dots,z_\ell],  \]
where $\ell$ denotes the number of $\wpt$-markings (and $\zpt$-markings) in $\Hup$,
and the curvature is specified by the matching $\Mup$.

Given an upper state $\x$, we define
\[ \Iup(\x)=\{1,\dots 2n-1\}\setminus \alpha(\x) \]
as in~\cite{HolKnot}; the following definition of $\bOut(\phi)$ is also
obtained from there:
\begin{defn}
  \label{def:Bout}
  Let $\phi\in\doms(\x,\y)$.  Define $b_0(\phi)$ to be the
  homogenenous element
  $b\in \BlgZ(n)\otimes \Field[w_1,z_1,\dots,w_\ell,z_\ell]$
  characterized by the following properties that
  \begin{itemize}
    \item 
      $\Iup(\x)\cdot b\cdot \Iup(\y) = b$; and
    \item for all $i=1,\dots,2n$,
      $\weight_i(b)$ is the average of the local multiplicities of $\phi$ in the
      two elementary domains adjacent to $Z_i$.
    \item The $w_i$ exponent of $b_0(\phi)$  is given by $n_{\wpt_i}(\phi)$
      and the $z_i$ exponent of $b_0(\phi)$ is given by $n_{\zpt_i}(\phi)$.
  \end{itemize}
  Let $\bOut(\phi)$ denote the induced element in
  $\Blg(n)\otimes\Field[w_1,z_1,\dots,w_\ell,z_\ell]$. 
\end{defn}

There is a $\Delta$ grading on the algebra, which is given by $-1$ times
the total weight of an algebra element. In particular,
\[ \Delta(w_i)=\Delta(z_i)=-1.\]
If $W$ is an oriented one-manifold with boundary $\{1,\dots,2n\}$
and $n+\ell$ components,
we have an induced Alexander grading in the algebra characterized by
\begin{align*}
  \Agr(b\otimes \prod_{i=1}^{\ell} w_i^{a_i} z_i^{b_i})&=\sum_{\{i,j\}\in\Matching}(\weight_i(b)-\weight_j(b))e_{\{i,j\}} \\
  \qquad & + \sum_{i=1} ^{\ell} (b_i-a_i) e_i,
\end{align*}
where $e_{\{i,j\}}$ is the generator of $H^1(W;\Z)$ corresponding to
the component of $W$ connecting $i$ to $j$, and 
$e_i$ is generator of $H^1(W;\Z)$ corresponding to the $i^{th}$ closed
component of $W$.

Let $\Mgr\colon \doms(\x,\y)\to \OneHalf\Z$ denote the Maslov index of
the homotopy class of flows from $\x$ to $\y$; as in~\cite{LipshitzCyl}, this is computed by the formula
\[ \Mgr(B)=e(B)+n_{\x}(B)+n_{\y}(B),\]
where $e(B)$ is the Euler measure of $B$.

\begin{prop}
  There is a function $\Mgr\colon \States(\Hup)\to \Z$
  uniquely characterized up to an overall constant by
  the property that
  \begin{equation}
    \label{eq:CharacterizeMgr}
    \Mgr(\x)-\Mgr(\y)=\Mgr(\phi)-\Delta(b_0(\phi)).
  \end{equation}
  Similarly, given an orientation the one-manifold $W$
  there is a function
  $\Agr\colon \States(\Hup)\to \OneHalf\Z^{n+\ell}$
  with components $\Agr_{\{i,j\}}$ corresponding to $\{i,j\}\in\Matching$
  characterized by
  \begin{equation}
        \label{eq:CharacterizeAgr}
    \Agr(\x)-\Agr(\y)=\Agr(b_0(\phi)).
  \end{equation}
\end{prop}

\begin{proof}
  This is a straightforward adaptation
  of~\cite[Proposition~\ref{HK:prop:DefineBigradingD}]{HolKnot}.  If
  $\phi, \phi'\in\doms(\x,\y)$, then 
  \[ \phi-\phi'=\sum_{i=1}^{\Lup} m_i \cdot A_i + \sum_{j=1}^{n+\Lup} n_j\cdot B_j,\] 
  where $A_i$ and $B_i$ are as in Definition~\ref{def:MarkedUpperDiagram}.
  To see Equation~\eqref{eq:CharacterizeMgr}, note that 
  for $\cald=A_i$ or $B_j$, we have
  $e(\cald)+n_{\x}(\cald)+n_{\y}(\cald)=2$, while
  each $\cald$ also has total weight $2$ ath the boundary.
  Equation~\eqref{eq:CharacterizeAgr} works similarly.
\end{proof}

\begin{equation}
  \label{eq:TypeDOperation}
  \delta^1(\x) = \sum_{\{\y\in \States, B\in\doms(\x,\y)\big|
    \Mgr(B)=1\}} \#\UnparModFlow^B(\x,\y) \cdot \bOut(B)\cdot \otimes
  \y.
\end{equation}

\begin{prop}
  The sum appearing on the right in Equation~\eqref{eq:TypeDOperation}
  is finite.
\end{prop}
\begin{proof}
  Given $\x,\y$, the set of $B\in\doms(\x,\y)$ with fixed $\Mgr(B)$
  is a finite set plus the addition of periodic domains
  By admissibility, only finitely many 
  of these elements have only positive local multiplicities.
\end{proof}

\begin{prop}
  \label{prop:CurvedDA}
  The map $\delta^1$ satisfies the curved type $D$ structure
  relation, with curvature $\mu_0=\sum_{\{r,s\}\in\Matching} U_r U_s$.
\end{prop}

\begin{proof}
  This works as in the proof
  of~\cite[Proposition~\ref{HK:prop:CurvedTypeD}]{HolKnot}, with a
  little modification.  In this case, there are two additional types
  of possible boundary degenerations, ones that contain exactly one
  $\wpt_i$ marking and exactly one $\zpt_i$ marking. In fact, for each $i$,
  there are exactly two such boundary degenerations: one with
  $\alpha$-boundary, and one with $\beta$-boundary. The contributions 
  to $\delta^1\circ \delta^1$ of these two boundary degenerations cancel.
\end{proof}

Let $\lsup{\cClgout}\Dmod(\Hup,J)$ denote the curved type $D$ structure defined above.

\begin{prop}
  If $J_0$ and $J_1$ are any two generic almost-complex structures,
  there quasi-isomorphism of curved, graded type $D$ structures
  over $\Clgout$
  \[\lsup{\cClgout}\Dmod(\Hup,J_0)\simeq \lsup{\cClgout}\Dmod(\Hup,J_1).\]
\end{prop}

\begin{proof}
  This follows exactly as
  in~\cite[Proposition~\ref{HK:prop:TypeDquasiIso}]{HolKnot}. 
\end{proof}

\subsection{Middle diagrams}

Fix a middle diagram
\begin{align*} \Hmid=(\Sigma_0,(\Zin_1,&\dots,\Zin_{2m}),(\Zout_1,\dots,\Zout_{2n}),
\{\alphain_1,\dots,\alphain_{2m-1}\},
\{\alphaout_1,\dots,\alphaout_{2n-1}\},\\
&\{\alpha^c_1,\dots,\alpha^c_{g}\},
\{\beta_1,\dots,\beta_{g+m+n+\Lmid-1}\},\{\wpt_1,\dots,\wpt_{\Lmid}\},\{\zpt_1,\dots,\zpt_{\Lmid}\}).
\end{align*}
Let $\Mmid$ be the induced $\beta$-matching as in Definition~\ref{def:BetaMatching}; and let $\MatchIn$ be a compatible matching on $\{1,\dots,2m\}$, in the sense
of Definition~\ref{def:Compatible}.

\begin{defn}
  Let $\Hmid$ be a middle diagram, equipped with a matching $\MatchIn$ on the
  incoming boundary components.
  Let 
  \[\Clgin(\Hmid)=\Clg(m)\otimes \Field[w_1,z_1,\dots,w_{\Lmid},z_{\Lmid}];
  \qquad \Clgout(\Hmid)=\Clg(n)\otimes \Field[w_1,z_1,\dots,w_{\Lmid},z_{\Lmid}].\]
\end{defn}

Our aim here is to define a curved type $DA$ bimodule associated to the middle diagram, denoted
$\lsup{\cClgout}\DAmod(\Hmid)_{\cClgin}$.

Each middle Heegaard state $\x$ determines two subsets
\[ \alphain(\x)\subset \{1,\dots,2m\}\qquad
{\text{resp.}}\qquad
\alphaout(\x)\subset \{1,\dots,2n\}\]
consisting of those $i\in\{1,\dots,2m\}$ resp.  $\{1,\dots,2n\}$
with $\x\cap \alphain_i\neq \emptyset$ resp 
$\x\cap\alphaout_i\neq\emptyset$. 
As an  $\Field$-vector space
$\DAmod(\Hmid)$ is spanned
spanned by
the middle Heegaard states of $\Hmid$.
Let 
\[ \IdempIn(\x)=\Idemp{\alphain(\x)}
\qquad{\text{and}}\qquad
\IdempOut(\x)=\Idemp{\{1,\dots,2n-1\}\setminus \alphaout(\x)}.\]
The 
$\IdempRing(n)-\IdempRing(m)$-bimodule structure is specified by
\begin{equation}  
\label{eq:DAoverIdempotents}
\IdempOut(\x)\cdot \x
\cdot \IdempIn(\x)=\x.
\end{equation}

 we need one more piece of data: an orientation 
$\orW$ on $W=\Wmid\cup\Win$.  Each boundary
component $\Zin_i$, $\Zout_j$, $\wpt_i$ or $\zpt_i$ 
of $\Hmid$ corresponds to some point
on $W$.

Orient $W$ from the $\wpt$ boundaries to the $\zpt$ boundaries.  (We choose
this convention to agree with~\cite{HolKnot}; it is opposite to the
orientation on the link as specified in
Section~\ref{subsec:HeegLinks}.)  There are other components of $W$
that connect various $\Zout$-boundaries. There are two kinds of $\Zin$
boundaries in $\Wmid$: those that are outwardly oriented in $\Wmid$,
and those that are inwardly oriented. Their corresponding orbits are
called {\em even} and {\em odd} respectively; and this partition of
the orbits into even and odd is called the {\em $\orW$-induced orbit
  marking}. The terminal boundary of any component of $W$ corresponds
either to a $\Zout$-boundary, or a $\zpt$-boundary. To each boundary
component in $\Zin_i$, which we think of as a point on $\orW$, we
associate an element $\zpt_j$ or $\Zout_j$, which is the terminal point
of the component of $\orW$ (with respect to its orientation) that
contains $\Zin_i$. We denote this $\tau(\Zin_i)$.

\begin{defn}
\label{def:AlgebraicPacket}
Recall that a set of Reeb chords 
$\{\rho_1,\dots,\rho_j\}$ in $\dIn \Hmid$
is called {\em algebraic} if
for any pair of distinct chords $\rho_a$ and $\rho_b$,
\begin{itemize}
  \item the chords $\rho_a$ and $\rho_b$ are on different boundary components $\Zout_i$ and $\Zout_j$,
  \item the initial points $\rho_a^-$ and $\rho_b^-$ are on different $\alpha$-curves; and
  \item the terminal points $\rho_a^+$ and $\rho_b^+$ are on different $\alpha$-curves.
\end{itemize}
\end{defn}
As explained in~\cite[Section~\ref{HK:subsec:AlgebraicConstraints}]{HolKnot},
algebraic packets determine algebra elements.

\begin{defn}
  \label{def:CompatiblePacketDA}
  Fix a Heegaard state $\x$ and a sequence $\vec{a}=(a_1,\dots,a_\ell)$ of pure
  algebra elements of $\Clgin(\Hmid)$. 
  A sequence of constraint packets $\rhos_1,\dots,\rhos_h$ is called
  \em{$(\x,\vec{a})$-compatible} if  there is a sequence 
  $1\leq k_1<\dots<k_\ell\leq k$ so that the following conditions hold:
  \begin{itemize}
  \item the constraint packets $\rhos_{k_i}$ consist of chords
    in $\Zin$, and they are algebraic, in the sense
    of Definition~\ref{def:AlgebraicPacket},
  \item  $\IdempIn(\x)\cdot \bIn(\rhos_{k_1})\otimes\dots\otimes \bIn(\rhos_{k_\ell})=
    {\mathbf I}(\x)\cdot a_1\otimes\dots\otimes a_{\ell}$,
    as elements of $\DAmod(\Hdown)\otimes \Clgin^{\otimes \ell}$
  \item
    for each $t\not\in \{k_1,\dots,k_\ell\}$,
    the constraint packet $\rhos_t$ is one 
    of the following  types:
    \begin{enumerate}[label=($DA\rhos$-\arabic*),ref=($DA\rhos$-\arabic*)]
    \item 
      \label{eq:OddOrbit}
      it is a singleton set  $\{\orb_i\}$,
      containing a single Reeb orbit $\orb_i$ 
      that covers the boundary component $\Zin_i$,
      which is {\em odd} for the $\orW$-induced orbit marking.
    \item 
      \label{eq:EvenOrbit}
      it consists of two elements
      $\{\orb_i,\longchord_j\}$, where $\orb_i$ is
      a Reeb orbit that covers a boundary component $\Zin_i$ 
      which is {\em even}, with
      $\{i,j\}\in\Matching$, and where
      $\longchord_j$ is one of the two Reeb chords
        that covers $\Zin_j$ with multiplicity one. 
    \item 
      \label{eq:OutOrbit}
      it is of the form $\{\orb_k\}$, where $\orb_k$ is the simple Reeb orbit
      around some component in $\Zout$
    \item 
      \label{eq:OutChord} it is of the form $\{\rho_j\}$, where $\rho_j$ is a Reeb chord
      of length $1/2$ supported in some $\Zout$.
  \end{enumerate}
\end{itemize}
Let $\llbracket \x,a_1,\dots,a_\ell\rrbracket$ 
be the set of all sequences of constraint packets $\rhos_1,\dots,\rhos_h$
that are $(\x,\vec{a})$-compatible.
\end{defn}

There is an algebra element $\bOut(B,\rhos_1,\dots,\rhos_h)$ defined
as follows.  We multiply $\bOut(B)$ (as defined in
Definition~\ref{def:Bout}) by a monomial in $\Clgout(\Hmid)$, determined as follows.
\begin{itemize}
\item
  For each packet $\rhos_i=\{\orb_i\}$ of Type~\eqref{eq:OddOrbit},
  multiply by the algebra element associated to the terminal point $\tau(\Zin_i)$.
  This terminal element $\tau(\Zin_i)$ can either be of the form $\Zout_j$, in which
  case the algebra element is $U_j$; or it can be some $\zpt_j$, in which case the algebra element is $z_j$.
\item
  For each packet $\rhos_i$ consisting of a single Reeb orbit
  that covers $\wpt_i$ or $\zpt_i$,
  multiply by a factor of $w_i$ resp. $z_i$ respectively.
\end{itemize}
(In particular, the packets of Type~\ref{eq:EvenOrbit} contribute a
factor of $1$.)

\begin{prop}
  \label{prop:FiniteSums}
  If $\Hmid$ is admissible, for each  $(\x,a_1,\dots,a_\ell)$,
  there are only finitely
  many choices of $(B,\rho_1,\dots,\rho_h)$ so that
  \begin{itemize}
  \item $\ind(B)=1$.
  \item The domain $B$ has only non-negative local multiplicities.
  \item $(\x,\rhos_1,\dots,\rhos_h)$ is $(\x,a_1,\dots,a_\ell)$-compatible.
  \end{itemize}
\end{prop}

\begin{proof}
  By the index requirement, the $\Delta$-grading on the output element
  is determined by $\x$, $\y$, and $(a_1,\dots,a_\ell)$.
  In particular, this gives an upper bound on the number of isolated odd orbits that
  are allowed among thet $\rhos_1,\dots,\rhos_h$, and also on the total weight
  of $B$ at $\Zout$.

  If both $B$ and $B'$ represent actions of $(\x,a_1,\dots,a_\ell)$
  with output $b\otimes \y$; and suppose that they have the same
  number of odd orbits.  Then, $B-B'$ is a periodic domain, in the
  sense of Definition~\ref{def:DA-Periodic-Domain}. Finiteness is now ensured by admissibility.
\end{proof}

\begin{defn}
  \label{def:DAmod}
  Fix the following data:
  \begin{itemize}
    \item an admissible middle diagram $\Hmid$ with matching $\Mmid$,
      with induced one-manifold $\Wmid$
    \item a compatible matching $\MatchIn$, with induced one-manifold $\Win$
    \item an orientation $\orW$ on the one-manifold
      $W=\Wmid\cup\Win$.
    \item a compatible almost-complex structure $J$ on $\Hmid$.
  \end{itemize}
  We abbreviate this data $(\Hmid,\MatchIn,\orW,J)$.
  Let $\DAmod$ be the associated  
  $\IdempRing(n)-\IdempRing(m)$-bimodule structure as specified in
  Equation~\eqref{eq:DAoverIdempotents}.
  For all $k\geq 0$, 
  define maps 
  $\delta^1_{k}\colon
  \DAmod\otimes \Clgin^{\otimes k}\to \DAmod$
  by 
  \begin{align}
\label{eq:DefDA-Action}
\delta^1_{k+1}&(\x,a_1,\dots,a_k)\\
&=
\sum_{
\left\{\begin{tiny}
\begin{array}{r}
\y\in\States \\
(\rhos_1,\dots,\rhos_h)\in \llbracket \x,a_1,\dots,a_k\rrbracket \\
B\in\pi_2(\x,\rhos_1,\dots,\rhos_h,\y)
\end{array}
\end{tiny}\Big| \ind(B,\rhos_1,\dots,\rhos_h)=1\right\}}
\!\!\!\!\!\!\!\!\!\!\!\!\!\!\!\!\!\!\!\!\!\!\!\!\#\UnparModFlow(\x,\y,\rhos_1,\dots\rhos_h)\cdot \bOut(B,\rhos_1,\dots,\rhos_h)\otimes \y.
\nonumber
\end{align}
\end{defn}

The sums in the above map are finite by Proposition~\ref{prop:FiniteSums}.

\begin{prop}
  \label{prop:DAmid}
  Let $(\Hmid,\MatchIn,\orW,J)$
  be the data required to define $\DAmod$, as in Definition~\ref{def:DAmod}.
  The $\IdempRing(2n)-\IdempRing(2m)$-bimodule
  $\DAmod(\Hmid)$, equipped the operations
  \[\delta^1_{\ell+1}\colon \DAmod(\Hmid)\otimes\Clgin^{\otimes \ell}\to \Clgout\otimes\DAmod(\Hmid)\]
  defined above endows $\DAmod(\Hmid)$ with the structure of 
  a curved $\Clg(n,\Mout)-\Clg(m,M)$ $DA$ bimodule,
  where $\Mout$ is the matching on $\{1,\dots,2n\}$ induced by 
  $\Mmid$ and $M$.
  This bimodule is also strictly unital and bounded.
  Moreover, for any two generic choices of $J$, the resulting
  curved $DA$ bimodules are homotopy equivalent.
\end{prop}

\begin{proof}
  The verification of the curved $DA$ bimodule relations follows as
  in~\cite[Proposition~\ref{HK:prop:DAmid}]{HolKnot}.  The key novelty
  is that now there are $\alpha$-boundary degenerations.  contribute
  terms of $w_i z_i\otimes \x$ to $\delta^1_1\circ \delta^1_1(\x)$.
  These cancel with the analogous terms coming from the
  $\beta$-boundary degenerations containing the basepoints $\wpt_i$.
  
  Strict unitality is obvious from the construction. Boundedness
  follows once again from admissibility.

  Varying $J$ induces homotopy equivalences by the usual continuation principle;
  compare for example~\cite[Section~\ref{HK:subsec:VaryCx}]{HolKnot}.
\end{proof}

Let $\lsup{\cClgout}\DAmod(\Hmid)_{\cClgin}$ denote the curved type $DA$ bimodule
associated to $\Hmid$ by the above procedure.

\begin{remark}
  The $DA$ bimodule depends on the data $(\Hmid,\MatchIn,\orW,J)$.
  The above theorem shows that its homotopy type is independent of the
  choice of $J$. The dependence on $\MatchIn$ is crucial. One could
  investigate its independence of the orientation $\orW$, but we do
  not need that in the sequel. Nonetheless, to shorten notation, we
  write simply $\lsup{\cClgout}\DAmod(\Hmid)_{\cClgin}$.
\end{remark}

\subsection{The pairing theorem}
\label{subsec:PairingTheorem}

\begin{defn}
  Fix the following:
  \begin{itemize}
  \item
    an marked upper diagram $\Hup$
  \item
    an marked middle diagram $\Hmid$
  \item an identification $\partial \Hup\cong \dIn\Hmid$.
  \end{itemize}
  This is called {\em compatibly gluable}
  if the matching $\Mmid$ is 
  compatible with the matching on $\dIn\Hmid$ induced by 
  $\Hup$ and the above identification. 
  Given boundary-identified diagrams $\Hmid$ and $\Hup$, their {\em gluing}
  $\Hmid\cup_{\dIn\Hmid\cong \partial \Hup}\Hup$ is naturally an upper diagram.
\end{defn}
  
If $\Hmid$ and $\Hup$ are admissible, their gluing is also admissible.

If $(\Hmid,\Hup,\partial\Hup\cong\dIn\Hmid)$ is
compatibly gluable, there is a one-to-one
correspondence between pairs of states $\x$ and $\y$, where $\x$ is an
partial Heegaard state for $\Hmid$ and $\y$ is an upper Heegaard state
for $\Hup$, and $\alpha(\x)=\{1,\dots,2n\}\setminus\alphaout(\y)$.

We have the following straightforward adaptation of the pairing theorem 
from~\cite[Theorem~\ref{HK:thm:PairDAwithD}]{HolKnot}
(compare also~\cite[Theorem~11]{Bimodules}):

\begin{thm} 
  \label{thm:PairDAwithD}
  Fix compatibly gluable diagrams $(\Hup,\Hmid,\dIn\Hmid\cong \partial\Hup)$,
  where $\Hup$ and $\Hmid$ are admissible.
  Let $\Clg_1=\Clg(\Hup)=\Clgin(\Hmid)$;
  $\Clg_2=\Clgout(\Hmid)$.

  Under the above hypotheses, there is a quasi-isomorphism
  of curved type $D$ structures
  \[\lsup{\Clg_2}\Dmod(\Hmid\#\Hup) 
  \simeq
  \lsup{\Clg_2}\DAmod(\Hmid)_{\Clg_1}\DT \lsup{\Clg_1}\Dmod(\Hup).\]
\end{thm}

\begin{proof}
  The proof of~\cite[Theorem~\ref{HK:thm:PairDAwithD}]{HolKnot} applies,
  after a few observations. Recall that in that earlier result, the
  tensor product is a curved type $D$ structure, with curvature
  $\sum_{\{i,j\}\in\Matching} U_i U_j$. The curvature term comes from
  $\beta$-boundary degenerations. The arguments from that proof also
  give rise to additional curvature terms $\sum_{i} w_i z_i$; but (as in the case
  of Proposition~\ref{prop:CurvedDA}) these terms are cancelled by
  $\alpha$-boundary degenerations.
\end{proof}

\subsection{Extending diagrams}

Fix the data $(\Hmid,\MatchIn, \orW,J)$ as in Definition~\ref{def:DAmod}
required to form the type DA bimodule $\lsup{\Clgout}\DAmod(\Hmid)_{\Clgin}$

Let $\Hmid$ be a marked middle diagram, 
and let $\HmidEx$ be an extended marked diagram
as in Definition~\ref{def:ExtendedMiddle}.

Define maps
\[ \delta^1_{1+k}\colon \DAmodEx(\HmidEx)\otimes 
{\Blgin}^{\otimes k} \to \Blgout\otimes\DAmodEx(\HmidEx).\]
in the obvious way.

Let 
${\widecheck \iota}\colon \Clgin\to \Blgin$ and
${\widehat \iota}\colon \Clgout\to \Blgout$
be the natural inclusion maps.

\begin{prop}
  \label{prop:ExtendDAPrecise}
  Fix the data $(\Hmid,\MatchIn, \orW,J)$ as in
  Definition~\ref{def:DAmod} required to form the type DA bimodule
  $\lsup{\Clgout}\DAmod(\Hmid)_{\Clgin}$, and let $\HmidEx$ be an
  extension of $\Hmid$. The type $DA$ bimodules associated to $\Hmid$
  and $\HmidEx$ are related by the formula
  \[ \lsup{\Blgout}[\iota]_{\Clgout}~
  \DT \lsup{\Clgout}\DAmod(\Hmid)_{\Clgin}=
  \lsup{\Blgout}\DAmodEx_{\Blgin}(\HmidEx)\DT~\lsup{\Blgin}[\iota]_{\Clgin}.\]
\end{prop}

\begin{proof}
  The proof is exactly as in the unmarked
  case~\cite[Proposition~\ref{HK:prop:ExtendDA}]{HolKnot}. 
\end{proof}

\subsection{Specializing to $w_i=z_i=0$}
 
Of course, all of the constructions described above can also be
specialized to $w_i=z_i=0$ for all $i$. The resulting modules are
denoted $\DmodHat(\Hup)$, $\DAmodHat(\Hmid)$, and $\AmodHat(\Hdown)$;
they are defined over the algebras where $w_i=z_i=0$; e.g.  if $\Hmid$
has $2m$ inputs and $2n$ outputs, then
\[ \DAmodHat(\Hmid)=\lsup{\Clg(n)}\DAmodHat(\Hmid)_{\Clg(m)}.\]
Explicitly, the holomorphic disks that contribute in the differential
of $\DAmodHat$ are required to have vanishing local multiplicity at
all $\wpt_i$ and $\zpt_i$; and moreover, they never count isolated
orbits.  (Isolated even orbits were not counted for $\DAmod$; isolated
odd orbits as in
Definition~\ref{def:CompatiblePacketDA}~\ref{eq:OddOrbit} contribute
$w_i$ factors in $\DAmod$, but we have set $w_i=0$ here.)

The above results can be readily specialized to this blocked theory.

For example, we can define $\DAmodExtHat(\HmidEx)$ to be the $w=z=0$
specialization of $\DAmodEx(\HmidEx)$.
Proposition~\ref{prop:ExtendDAPrecise} implies then that
\[ \lsup{\Blg(n)}[\iota]_{\Clg(n)}~ \DT
\lsup{\Clg(n)}\DAmodHat(\Hmid)_{\Clg(m)}=
\lsup{\Blg(n)}\DAmodExtHat_{\Blg(m)}(\HmidEx)\DT~\lsup{\Blg(m)}[\iota]_{\Clg(m)}.\]
Also,  Theorem~\ref{thm:PairDAwithD} has the following specialization:

\begin{thm} 
  \label{thm:PairDAwithDhat}
  Fix compatibly gluable diagrams
  $(\Hup,\Hmid,\dIn\Hmid\cong \partial\Hup)$, where $\Hup$ and $\Hmid$
  are admissible, where Let $\Clg_1=\Clg(m)$; $\Clg_2=\Clg(n)$,
  provided that $\Hmid$ has $2m$ incoming boundary circles and $2n$
  outgoing ones. There is a quasi-isomorphism of curved type
  $D$ structures
  \[\lsup{\Clg_2}\DmodHat(\Hmid\#\Hup) 
  \simeq
  \lsup{\Clg_2}\DAmodHat(\Hmid)_{\Clg_1}\DT \lsup{\Clg_1}\DmodHat(\Hup),\]
\end{thm}

\begin{proof}
  This follows immediately from Theorem~\ref{thm:PairDAwithD}, after
  setting $w_i=z_i=0$.
\end{proof}

\section{Modules for a marked minimum}
\label{sec:AlgMarkedMin}

Our aim here is to define algebraically type $DA$ bimodules.
These will come in two versions:
\begin{itemize}
\item $\lsup{\Blg(n)} {\MarkedMinHat}_{\Blg(n+1)}$
\item $\lsup{\Blg(n)} {\MarkedMin}_{\Blg(n+1)}$ with base ring $\Field[w,z]$, whose $w=z=0$
  specialization is the above module.
\end{itemize}
These will be used to define also curved variants
$\lsup{\cBlg_2} {\MarkedMinHat}_{\cBlg_1}$
and $\lsup{\cBlg_2} {\MarkedMin}_{\cBlg_1}$
(cf. Proposition~\ref{prop:CurvedDAmin}).

We shall also define another module
$\lsup{\nDuAlg_1}\nMarkedMin_{\nDuAlg_2}$
with the property that
\[ \lsup{\cBlg_2}\nMarkedMin_{\cBlg_1}
\DT \lsup{\cBlg_1,\nDuAlg_1}\CanonDD\simeq 
\lsup{\nDuAlg_1}\nMarkedMin_{\nDuAlg_2}
\DT \lsup{\nDuAlg_2,\cBlg_2}\CanonDD \]
(cf. Proposition~\ref{prop:CommuteWithCanon}).

Their relationship with the holomorphicall defined modules will be
established in Section~\ref{subsec:ComputeMarkedMin}.

\subsection{Modules in the algebraically specialized case}

Consider the subring $\IdempRing_{\geq k}\subset \IdempRing(n+1)$
of the idempotent ring of $\Blg(n+1)$.

Let $\Matching$ be a matching on $\{1,\dots,2n+2\}$ with $\{1,2\}\in
\Matching$.  Let $\Matching'$ be the corresponding matching on
$\{1,\dots,2n\}$ with $\{i,j\}\in \Matching'$ if and only if
$\{i+2,j+2\}\in\Matching$.

The idempotent states $\x$ for
$\Blg_1=\Blg(n+1)$ are called {\em preferred} if
\[ 1\leq |\x\cap \{0,1,2\}|\leq 2
\qquad\text{and}\qquad
\x\neq \{0,1\}
\]
Given a preferred idempotent state $\x$ for
$\Blg_1$, the corresponding idempotent $\psi(\x)$ for $\Blg(n)=\Blg_2$
contains $0$ precisely when $|\x\cap \{0,1,2\}|=2$; and for $i>0$, $i\in
\psi(\x)$ precisely when $i+2\in \x$.

We construct now a type $DA$ bimodule
$\lsup{\Blg(n)}{\MarkedMinHat}_{\Blg(n+1)}$, as follows.

As a right $\IdempRing(n+1)$-module, $\MarkedMinHat$ splits
as a direct sum of six modules 
\[ X_0\oplus X_1 \oplus X_2\oplus Y_0\oplus Y_1\oplus Y_2,\]
where $X_i\cong Y_i$ is generated by the preferred idempotents in $I_{\geq i}$.
The left $\IdempRing(n)$-module action is specified by
\[ X\cdot \Idemp{\x}=\Idemp{\psi(\x)}\cdot X \cdot \Idemp{\x}\]
for any $X\in \MarkedMinHat$.

For example, $X_0$ splits naturally into two summands, according to
the whether or not $0$ appears in the left idempotent: a summand where
$0$ appears in the left idempotent and both $0$ and $2$ (but not $1$
appear in the right idempotent; another one where $0$ does not appear
in the left idemptent and $0$ (but not $1$ or $2$) appears in the
right idempotent.

The action by $\Blg(n+1)$ is specified using
Figure~\ref{fig:MarkedMinHat}. That figure appears to specify a
bimodule with incoming algebra $\Blg(1)$ and outgoing algebra $\Blg(0)$. We extend this to a bimodule $\lsup{\Blg(n,\Matching')}{\MarkedMinHat}_{\Blg(n+1,\Matching)}$ as follows.

\begin{figure}
  \[
    \begin{tikzpicture}[scale=1.8]
    \node at (0,3.5) (A2) {$X_1$} ;
    \node at (-2,3) (A3) {$X_2$} ;
    \node at (2,3) (A1) {$X_0$} ;
    \node at (-1,2) (B1) {$Y_0$} ;
    \node at (1,2) (B3) {$Y_2$} ;
    \node at (0,1) (B2) {$Y_1$} ;
    \draw[->](A2) to node[above,sloped] {\tiny{$(U_1,U_2)+(U_2,U_1)$}} (B2);
    \draw[->] (A2)  to node[above,sloped] {\tiny{$R_2$}}  (A3)  ;
    \draw[->] (A2)  to node[above,sloped] {\tiny{$L_1$}}  (A1)  ;
    \draw[->] (A1)  to node[above,sloped] {\tiny{$(R_1,R_2)$}}  (B3)  ;
    \draw[->] (B3)  to node[above,sloped] {\tiny{$L_2$}}  (B2)  ;
    \draw[->] (A3) to node[above,sloped]{\tiny{$(L_2,L_1)$}} (B1) ;
    \draw[->] (B1) to node[above,sloped]{\tiny{$R_1$}} (B2);
    \draw[->] (A2) to node[above, sloped]{\tiny{$(U_2,L_1)$}} (B1) ;
    \draw[->] (A2) to node[above, sloped]{\tiny{$(U_1,R_2)$}} (B3) ;
    \draw[->] (A3) [bend right=30] to node[below, sloped] {\tiny{$(L_2,U_1)$}} (B2) ;
    \draw[->] (A1) [bend left=30] to node[below, sloped] {\tiny{$(R_1,U_2)$}} (B2) ;
  \end{tikzpicture}
    \]
    \caption{Actions on $\MarkedMinHat$}
      \label{fig:MarkedMinHat}
\end{figure}

If $b=\Idemp{\x}\cdot b\cdot \Idemp{\y}\in I_{\geq 1}\cdot
\Blg(n+1)\cdot I_{\geq 1}$ has $\weight_1(b)=\weight_2(b)=0$, then
let $\Psi(b)=b'$ denote the algebra element
$b'=\Idemp{\psi(\x)}\cdot b'\cdot \Idemp{\psi(\y)}$
with $\weight_i(b')=\weight_{i+2}(b)$. 
Define $\Psi(b)$ similarly for $b\in I_{\geq 2}\cdot \Blg(n+1)\cdot I_{\geq 2}$.

For the extension, define
$\delta^1_2(X,b)=\Psi(b)\otimes X$ for all $X\in \MarkedMinHat$. The arrows above are similarly extended; e.g.
\[ \delta^1_3(X_2, L_2 b_1, L_1 b_2)=\Psi(b_1)\cdot \Psi(b_2)\otimes
Y_0.\] In particular, $\delta^1_k\equiv 0$ for $k\geq 4$.

\begin{defn}
  Let $\Alg$, $\Blg$ be algebras and $\lsup{\Blg}X_{\Alg}$ a type $DA$
  bimodule. Fix $a\in \Alg$ and $b\in\Blg$. We say that $X$ is
  {\em{$a$-$b$-equivariant}} if the following conditions hold:
  \begin{itemize}
    \item $\delta^1_1(x,a)=b\otimes x$ for all $x\in X$
    \item $\delta^1_{k+1}(x,a_1,\dots,a_k)=0$ if for all $k>1$
      if $a_i=a$ for some $i$.
  \end{itemize}
  We say that $X$ is {\em annihilated by} $a$  if $X$ is $a$-$0$-equivariant.
\end{defn}

\begin{prop}
  \label{prop:MarkedMinHat}
  The object $\lsup{\Blg_2}{\MarkedMinHat}_{\Blg_1}$ is 
  a type $DA$ bimodule,
  which is $U_r\cdot U_s$-$U_{r+2} U_{s+2}$ equivariant for all
  $r,s\in\{1,\dots,2n\}$; and which is annihilated by $U_1 U_2$.
\end{prop}

\begin{proof}
  These are straightforward computations.
\end{proof}

\subsection{The algebraically unspecialized case}

We can extend $\lsup{\Blg_2}\MarkedMinHat_{\Blg_1}$
to a module $\lsup{\Blg_2}\MarkedMin_{\Blg_1}$
over $\Field[w,z]$. Generators for the module are the same as before,
but now actions are specified in the following diagram, where $k, \ell$ are arbitrary non-negative integers.
 \begin{equation}
    \begin{tikzpicture}[scale=1.8]
    \node at (0,3.5) (A2) {$X_1$} ;
    \node at (-2.5,2.8) (A3) {$X_2$} ;
    \node at (2.5,2.8) (A1) {$X_0$} ;
    \node at (-1.5,0) (B1) {$Y_0$} ;
    \node at (1.5,0) (B3) {$Y_2$} ;
    \node at (0,-2) (B2) {$Y_1$} ;
    \draw[->] (A2) [bend right=5] to node[above,sloped] {\tiny{$z^\ell \otimes R_2 U^\ell$}}  (A3)  ;
    \draw[->] (A3) [bend right=5] to node[below,sloped] {\tiny{$z^{\ell+1} \otimes L_2 U^\ell$}}  (A2)  ;
    \draw[->] (A2) [bend left=5] to node[above,sloped] {\tiny{$w^k \otimes L_1 U_1^k$}}  (A1)  ;
    \draw[->] (A1) [bend left=5] to node[below,sloped] {\tiny{$w^{k+1} \otimes R_1 U_1^k$}}  (A2)  ;
    \draw[->] (A2) [bend left=5] to node[above,sloped, pos=.2] {\tiny{$w^k z^\ell \left((U_1^{k+1},U_2^{\ell+1}) + (U_2^{\ell+1},U_1^{k+1})\right)$}} (B2);
    \draw[->] (B2) [bend left=5] to node[below,sloped,pos=.2]{\tiny{$w z$}} (A2);
    \draw[->] (A1)  to node[above,sloped] {\tiny{$w^{k}z^{\ell}(R_1 U_1^k,R_2 U_2^\ell)$}}  (B3)  ;
    \draw[->] (B3) [bend right=5] to node[above,sloped] {\tiny{$z^\ell \otimes L_2 U_2^\ell$}}  (B2)  ;
    \draw[->] (B2) [bend right=5] to node[below,sloped] {\tiny{$z^{\ell+1} \otimes R_2 U_2^{\ell}$}}  (B3)  ;
    \draw[->] (A3) to node[above,sloped]{\tiny{$w^{k} z^\ell\otimes (L_2 U_2^\ell,L_1 U_1^k)$}} (B1) ;
    \draw[->] (B1) [bend left=5] to node[above,sloped]{\tiny{$w^k\otimes R_1 U_1^k$}} (B2);
    \draw[->] (B2) [bend left=5] to node[below,sloped]{\tiny{$w^{k+1}\otimes L_1 U_1^k$}} (B1);
    \draw[->] (A2) to node[above, sloped,pos=.4]{\tiny{$w^k z^\ell \otimes (U_2^{\ell+1},L_1 U_1^k)$}} (B1) ;
    \draw[->] (A2) to node[above, sloped,pos=.4]{\tiny{$w^k z^\ell (U_1^{k+1},R_2 U_2^\ell)$}} (B3) ;
    \draw[->] (A3) [bend right=30] to node[below, sloped] {\tiny{$w^k z^\ell(L_2 U_2^\ell,U_1^{k+1})$}} (B2) ;
    \draw[->] (A1) [bend left=30] to node[below, sloped]  {\tiny{$w^k z^\ell(R_1 U_1^k, U_2^{\ell+1})$}} (B2) ;
    \draw[->] (B3) [bend left=15] to node[below,sloped,pos=.6] {\tiny{$w$}} (A3) ;
    \draw[->] (B1) [bend right=15] to node[below,sloped,pos=.6] {\tiny{$z$}} (A1) ;
    \draw[->] (A2) [loop above] to node[above,sloped] {\tiny{$w^k\otimes U_1^k+ z^{\ell+1} \otimes U_2^{\ell+1}$}} (A2);
    \draw[->] (B2) [loop below] to node[below,sloped] {\tiny{$w^k\otimes U_1^k+ z^{\ell+1} \otimes U_2^{\ell+1}$}} (B2);
    \draw[->] (A1) [loop above] to node[above,sloped] {\tiny{$w^k \otimes U_1 ^k$}} (A2);
    \draw[->] (B1) [loop right] to node[above,sloped] {\tiny{$w^k \otimes U_1 ^k$}} (B1);
    \draw[->] (A3) [loop above] to node[above,sloped] {\tiny{$z^\ell \otimes U_2 ^\ell$}} (A3);
    \draw[->] (B3) [loop left] to node[above,sloped] {\tiny{$z^\ell \otimes U_2 ^\ell$}} (B3);
  \end{tikzpicture}
\end{equation}

The arrows are to be interpreted as before; e.g.
\[ \delta^1_3(X_2, L_2 U_2\cdot b_1,L_1 U_1^2)=w^2 z \cdot
\Psi(b_1)\cdot \Psi(b_2)\otimes Y_0.\] Again, $\delta^1_k\equiv 0$ for
$k\geq 4$.

\begin{prop}
  \label{prop:MarkedMin}
  The object $\lsup{\Blg_2[w,z]}{\MarkedMin}_{\Blg_1}$ is a type $DA$
  bimodule, which is annihilated by $U_1 U_2$, and which is $U_r
  U_s$-$U_{r+2} U_{s+2}$-equivariant for all $r,s\in\{1,\dots,2n\}$.
  Specializing $\MarkedMin$ to $w=z=0$
  gives $\lsup{\Blg_2}{\MarkedMinHat}_{\Blg_1}$.
\end{prop}

\subsection{Curved modules}
\label{subsec:CurvedModules}

  Fix any matching $\Matching_2$ on $\{1,\dots,2n\}$, and let
  $\Matching_1$ be the matching matching on $\{1,\dots,2n+2\}$
  specified as follows:
  \begin{itemize}
  \item $\{1,2\}\in\Matching_1$
  \item For all $i,j\in\{1,\dots,2n\}$,
    $\{i,j\}\in\Matching_2\iff\{i+2,j+2\}\in\Matching_1$.
  \end{itemize}
  Let $\cBlg_1=\cBlg(n+1,\Matching_1)$,
  $\cBlg_2=\cBlg(n,\Matching_2)$,
  and $\cBlg_2[w,z]$ denote $\cBlg_2$ where the algebra is extended
  by two additional variables $w$ and $z$.

\begin{prop}
  \label{prop:CurvedDAmin}
  For $\cBlg_1$ and $\cBlg_2$ as above, the operations
  $\delta^1_{1+k}$ endow $\lsup{\cBlg_2}{\MarkedMinHat}_{\cBlg_1}$ the
  structure of a curved type $DA$ bimodule.  More generally, the
  operations specified above
  $\lsup{\cBlg_2[w,z]}{\MarkedMin}_{\cBlg_1}$ is a curved $DA$
  bimodule, whose $w=z=0$ specialization is $\MarkedMinHat$.
\end{prop}

\begin{proof}
  Proposition~\ref{prop:MarkedMinHat} ensures that $\MarkedMinHat$ is
  $\mu_0^{\Matching_1}$-$\mu_0^{\Matching_2}$-equivariant type $DA$
  bimodule. But such a bimodule is simply a curved type $DA$ bimodule.
  For $\MarkedMin$, the result follows from
  Proposition~\ref{prop:MarkedMin}.
\end{proof}

\subsection{Commuting with the canonical bimodule}

Let $\nDuAlg_1=\nDuAlg(n+1,\Matching_1)$ and
$\nDuAlg_2=\nDuAlg(n,\Matching_2)$, with $\Matching_1$ and
$\Matching_2$ as in Section~\ref{subsec:CurvedModules}.  Our aim here
is to construct a type $DA$ bimodule
$\lsup{\nDuAlg_1}\nMarkedMin_{\nDuAlg_2}$ which is dual to
$\MarkedMin$.

As a left $\IdempRing_{2n+2,n+2}$ module
$\nMarkedMin$ splits into summands
\begin{equation}
  \label{eq:SixTypes}
  \nMarkedMin=
  X_0'\oplus X_1' \oplus X_2'\oplus Y_0'\oplus Y_1'\oplus Y_2',
\end{equation}
called {\em types}; 
where the idempotents of each summand is complementary to the corresponding summand in $\MarkedMin$; i.e.
\[ \begin{array}{lll}
  X_1'=X'_{\{0,2\}}\oplus X'_{0} & 
  X_0'= X'_{\{1,2\}}\oplus X'_{\{1\}} &
  X_2'= X'_{\{0,1\}} \\
  Y_0'=Y'_{\{1,2\}}\oplus Y'_{\{1\}} &
  Y_1'=Y'_{\{0\}}\oplus Y'_{\{0,2\}} &
  Y_2'=Y'_{\{0,1\}}
  \end{array}
\]
where here the subscript indicates the idempotents; e.g. as a left
$\IdempRing(2n+2,n+2)$-module, we have
\[ X'_{\{1,2\}}\cong \IdempRing(2n+2,n+2)\cdot \left(\sum_{\{\x\big|\x\cap\{0,1,2\}=\{1,2\}\}} \Idemp{\x}\right).\]

Idempotent states $\x$ for $\nDuAlg$ are called {\em preferred} if
\[ 1\leq |\x\cap \{0,1,2\}|\leq 2
\]
Given a preferred idempotent state $\x$ for $\Blg_1(2n+2,n+2)$, there is a
corresponding idempotent $\varphi(\x)$ for $\Blg(2n,n+1)$, which contains
$0$ precisely when $|\x\cap \{0,1,2\}|=2$; and for $i>0$, $i\in
\varphi(\x)$ precisely when $i+2\in \x$.  The bimodule structure is
now specified by requiring
\[ \Idemp{\x}\cdot P = P \cdot \Idemp{\varphi(\x)}.\]

Define 
\[ \delta^1_1 \colon \nMarkedMin\to \nDuAlg_1 \otimes \nMarkedMin \]
as in Figure~\ref{fig:nMarkedMin}.

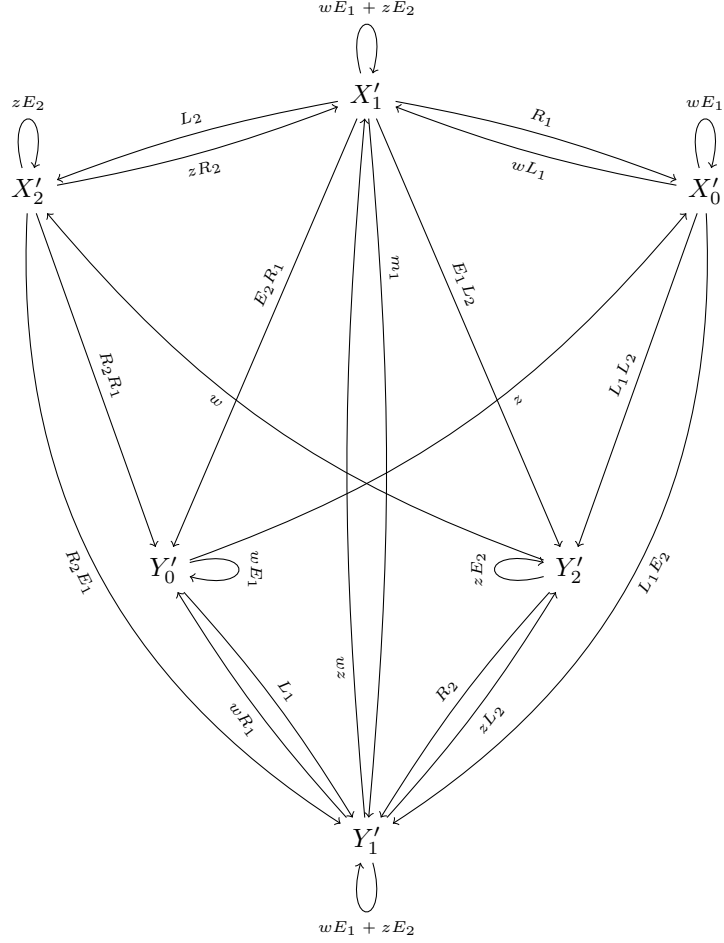
\begin{figure}
  \[
    \begin{tikzpicture}[scale=1.8]
    \node at (0,3.5) (A2) {$X'_1$} ;
    \node at (-2.5,2.8) (A3) {$X'_2$} ;
    \node at (2.5,2.8) (A1) {$X'_0$} ;
    \node at (-1.5,0) (B1) {$Y'_0$} ;
    \node at (1.5,0) (B3) {$Y'_2$} ;
    \node at (0,-2) (B2) {$Y'_1$} ;
    \draw[->] (A2) [bend right=5] to node[above,sloped] {\tiny{$ L_2$}}  (A3)  ;
    \draw[->] (A3) [bend right=5] to node[below,sloped] {\tiny{$z R_2$}}  (A2)  ;
    \draw[->] (A2) [bend left=5] to node[above,sloped] {\tiny{$R_1$}}  (A1)  ;
    \draw[->] (A1) [bend left=5] to node[below,sloped] {\tiny{$w L_1$}}  (A2)  ;
    \draw[->] (A2) [bend left=5] to node[above,sloped, pos=.2] {\tiny{$m_1$}} (B2);
    \draw[->] (B2) [bend left=5] to node[below,sloped,pos=.2]{\tiny{$w z$}} (A2);
    \draw[->] (A1)  to node[above,sloped] {\tiny{$L_1 L_2$}}  (B3)  ;
    \draw[->] (B3) [bend right=5] to node[above,sloped] {\tiny{$R_2$}}  (B2)  ;
    \draw[->] (B2) [bend right=5] to node[below,sloped] {\tiny{$z L_2$}}  (B3)  ;
    \draw[->] (A3) to node[above,sloped]{\tiny{$R_2 R_1$}} (B1) ;
    \draw[->] (B1) [bend left=5] to node[above,sloped]{\tiny{$ L_1$}} (B2);
    \draw[->] (B2) [bend left=5] to node[below,sloped]{\tiny{$w R_1$}} (B1);
    \draw[->] (A2) to node[above, sloped,pos=.4]{\tiny{$E_2 R_1$}} (B1) ;
    \draw[->] (A2) to node[above, sloped,pos=.4]{\tiny{$E_1 L_2$}} (B3) ;
    \draw[->] (A3) [bend right=30] to node[below, sloped] {\tiny{$R_2 E_1$}} (B2) ;
    \draw[->] (A1) [bend left=30] to node[below, sloped]  {\tiny{$L_1 E_2$}} (B2) ;
    \draw[->] (B3) [bend left=15] to node[below,sloped,pos=.6] {\tiny{$w$}} (A3) ;
    \draw[->] (B1) [bend right=15] to node[below,sloped,pos=.6] {\tiny{$z$}} (A1) ;
    \draw[->] (A2) [loop above] to node[above,sloped] {\tiny{$w E_1+ z E_2$}} (A2);
    \draw[->] (B2) [loop below] to node[below,sloped] {\tiny{$w E_1 + z E_2$}} (B2);
    \draw[->] (A1) [loop above] to node[above,sloped] {\tiny{$w E_1$}} (A2);
    \draw[->] (B1) [loop right] to node[above,sloped] {\tiny{$w E_1$}} (B1);
    \draw[->] (A3) [loop above] to node[above,sloped] {\tiny{$z E_2$}} (A3);
    \draw[->] (B3) [loop left] to node[above,sloped] {\tiny{$z E_2$}} (B3);
  \end{tikzpicture}
    \]
    \caption{Specifying $\delta^1_1$ in $\nMarkedMin$}
    \label{fig:nMarkedMin}
\end{figure}

Define 
\[ \delta^1_2 \colon \nMarkedMin\otimes \nDuAlg_2\to \nDuAlg_1 \otimes \nMarkedMin \]
as follows.
Given $P\in \MarkedMin$ in one of the six
summands from Equation~\eqref{eq:SixTypes}, with
$P=\Idemp{\x}\cdot P$,  and given
$b=\Idemp{\varphi(\x)}\cdot b \cdot \Idemp{\y}\in \Blg(2n,n+1)\subset\nDuAlg_2$,
let $\Phi(b)=b'\in\Blg(2n+2,n+2)\subset \nDuAlg_1$ be the algebra element
$b'=\Idemp{\x}\cdot b'\cdot \Idemp{\y}$ with weight
$\weight_i(b')=\weight_{i-2}(b)$ and $\weight_0(b')=\weight_1(b')$.
Then,
\[\delta^1_2(X,b)=\Phi(b)\cdot P',\]
where $P'=\Idemp{\y}\cdot P'$ is in the same type as $X$.
We extend this to an action of all of $\nDuAlg_2$ by requiring 
\[ \delta^1_2(X,E_i\cdot b)=E_{i+2}\cdot \delta^1_2(X,b).\]

\begin{prop}
  \label{prop:CommuteWithCanon}
  The operations give $\nMarkedMin$ the structure of a type $DA$
  bimodule $\lsup{\nDuAlg_1}{\nMarkedMin}_{\nDuAlg_2}$.  This bimodule
  is related to $\lsup{\cBlg_2}\MarkedMin_{\cBlg_1}$ by the identity
  \[ \lsup{\nDuAlg_1}{\nMarkedMin}_{\nDuAlg_2}\DT~ \lsup{\nDuAlg_2,\cBlg_2}\CanonDD\simeq~ 
  \lsup{\nDuAlg_2}{\MarkedMin}_{\nDuAlg_1}\DT~ \lsup{\nDuAlg_1,\cBlg_1}\CanonDD;\]
  similarly, 
  \[ \lsup{\nDuAlg_1}{\nMarkedMinHat}_{\nDuAlg_2}\DT~ \lsup{\nDuAlg_2,\cBlg_2}\CanonDD\simeq~
  \lsup{\nDuAlg_2}{\MarkedMinHat}_{\nDuAlg_1}\DT~ \lsup{\nDuAlg_1,\cBlg_1}\CanonDD.\]
\end{prop}
\begin{proof}
  These are all straightforward computations.
\end{proof}

\section{Holomorphic computations}
\label{subsec:ComputeMarkedMin}

\begin{figure}[h]
 \centering
 \input{MarkedMin.pstex_t}
 \caption{{\bf Heegaard diagram for a marked minimum.}}
 \label{fig:MarkedMin}
 \end{figure}

The diagram in Figure~\ref{fig:MarkedMin} shows a Heegaard diagram for
a marked minimum. The left portion of this diagram (redrawn on the
left in Figure~\ref{fig:MarkedMinExt}) is stabilized as shown on the
right in Figure~\ref{fig:MarkedMinExt}.  Rather than working directly
with this diagram, we will find it convenient to work with the
isotopic diagram shown in Figure~\ref{fig:NewMarkedMin}.  Since these
diagrams are isotopic (and both are admissible), it follows that their
associated type $DA$ bimodules are homotopy equivalent.  (This
homotopy equivalence can be constructed by a continuation map, similar
to the proof of $J$-invariance in Proposition~\ref{prop:DAmid}.)
We denote this latter
diagram $\HmidEx$.

\begin{figure}[h]
 \centering
 \input{MarkedMinExt.pstex_t}
 \caption{{\bf Heegaard diagram for a marked minimum.}}
 \label{fig:MarkedMinExt}
 \end{figure}

\begin{figure}[h]
 \centering
 \input{NewMarkedMin.pstex_t}
 \caption{{\bf Heegaard diagram for a marked minimum.}}
 \label{fig:NewMarkedMin}
 \end{figure}

After making some conformal
choices, compute enough of that bimodule to compute the associated
type $DD$-bimodule
$\lsup{\cBlg_2}\DAmodEx(\HmidE)_{\cBlg_1}\DT \lsup{\Blg_1,\nDuAlg}\CanonDD$.
Indeed, the main result of this section is the following
\begin{prop}
  \label{prop:ComputeDD}
  For a suitable choice of almost-complex structure on $\HmidE$, there
  is an identification $\lsup{\cBlg_2}\DAmodEx(\HmidEx)_{\cBlg_1} \DT
  \lsup{\cBlg_1,\nDuAlg_1} \CanonDD \simeq
  \lsup{\cBlg_2}\MarkedMin_{\cBlg_1}\DT \lsup{\cBlg_1,\nDuAlg_1}
  \CanonDD$, where $\MarkedMin$ is the algebraically defined bimodule
  constructed in Section~\ref{sec:AlgMarkedMin}.
\end{prop}

Before turning to the proof, we make a few notational remarks about
the statement.  Note that $\HmidEx$ is an extended diagram with
$\ell=1$, so its output algebra is $\Clg(n)\otimes
\Field[w_1,z_1]$. Correspondingly, the output algebra $\MarkedMin$ as
defined in Section~\ref{sec:AlgMarkedMin} was $\Clg(n)\otimes
\Field[w,z]$. In the above identification of modules, we are using an
isomorphism of algebras identifying $w_1$ and $z_1$ with $w$ and $z$
respectively.

\subsection{The blocked case}
In this section, we prove the following variant of
Proposition~\ref{prop:ComputeDD}, specialized to $w=z=0$.
(Note that this
special case is sufficient to compute $\HFLa$ for links which, in
turn, determines the Thurston polytope of the underlying link.)

\begin{prop} 
  \label{prop:SpecialComputeDD}
  For a suitable choice of almost-complex structure on $\HmidE$, there
  is a homotopy equivalence
  \begin{equation}
    \label{eq:ComputeBlockedMin}
    \left(\frac{\lsup{\Blg_2}\DAmodEx(\HmidE)_{\Blg_1}}{w=z=0}\right) \DT
    \lsup{\Blg_1,\DuAlg_1} \CanonDD \simeq
    \lsup{\Blg_2}\MarkedMinHat_{\Blg_1}\DT \lsup{\Blg_1,\DuAlg_1}
    \CanonDD.
\end{equation}
\end{prop}

\begin{proof}
  Our aim here is to prove this specialization here.  
First note that there are $12$ generator types locally. We denote them
\[
\begin{array}{lllll}
\XX_2& \YY_2& \lsup{0}\XX_{0, 2}& \lsup{0}\YY_{0,2}, & \Ax_2\\
\BB_2& \Ax_1, & \BB_1& \lsup{0}\Ax_{1,2}& \lsup{0}\BB_{1,2}. 
\end{array}\]

We claim there is an almost-complex structure
for which the moduli has the actions displayed in
in Figure~\ref{fig:MarkedMinHat}.
\begin{figure}
  \[
    \begin{tikzpicture}[scale=1.8]
    \node at (0,3.5) (A2) {$\Ax_1,~\lsup{0}\Ax_{1,2}$} ;
    \node at (-2,3) (A3) {$\Ax_2$} ;
    \node at (2,3) (A1) {$\XX_0,~\lsup{0}\XX_{0,2}$} ;
    \node at (-1,2) (B1) {$\YY_0,~\lsup{0}\YY_{0,2}$} ;
    \node at (1,2) (B3) {$\XX_2$} ;
    \node at (0,1) (B2) {$\BB_1,~\lsup{0}\BB_{1,2}$} ;
    \node at (-3,2) (T) {$\BB_2$};
    \node at(-2,1) (B) {$\YY_2$};
    \draw[->](B) to node[above,sloped] {\tiny{$\delta$}} (T) ;
    \draw[->](A3) to node[above,sloped] {\tiny{$\delta$}} (T);
    \draw[->](A2) to node[above,sloped] {\tiny{$(U_1,U_2)+(U_2,U_1)$}} (B2);
    \draw[->] (A2)  to node[above,sloped] {\tiny{$R_2$}}  (A3)  ;
    \draw[->] (A2)  to node[above,sloped] {\tiny{$L_1$}}  (A1)  ;
    \draw[->] (A1)  to node[above,sloped] {\tiny{$(R_1,R_2)$}}  (B3)  ;
    \draw[->] (B3)  to node[above,sloped] {\tiny{$L_2$}}  (B2)  ;
    \draw[->] (A3) to node[above,sloped]{\tiny{$(L_2,L_1)$}} (B1) ;
    \draw[->] (B1) to node[above,sloped]{\tiny{$R_1$}} (B2);
    \draw[->] (A2) to node[above, sloped]{\tiny{$(U_2,L_1)$}} (B1) ;
    \draw[->] (A2) to node[above, sloped]{\tiny{$(U_1,R_2)$}} (B3) ;
    \draw[->] (A3) [bend right=30] to node[below, sloped] {\tiny{$(L_2,U_1)$}} (B2) ;
    \draw[->] (A1) [bend left=30] to node[below, sloped] {\tiny{$(R_1,U_2)$}} (B2) ;
  \end{tikzpicture}
    \]
    \caption{Some actions on $\DAmodExtHat(\HmidEx)$}
      \label{fig:MarkedMinHatHol}
\end{figure}
(In that figure, arrows labelled by $\delta$ represent $\delta^1_1$ actions.)
Moreover, we show that there are no other actions starting at $\YY_2$ or $\BB_2$;
i.e. $\YY_2$ and $\BB_2$ generate a type $DA$ submodule.

To this end, label the domains in the Heegaard diagram as shown in
Figure~\ref{fig:MarkedMinDoms}. 

The rectangle $\cald_3$ represents the differential
\[\begin{tiny}\begin{CD}
  \YY_2@>{\delta}>{\cald_3}> \BB_2.
\end{CD}\end{tiny}\]

\begin{figure}[h]
 \centering
 \input{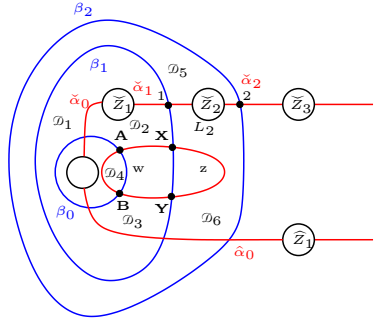}
 \caption{{\bf Heegaard diagram for a marked minimum.}  We have labelled various
   domains.}
 \label{fig:MarkedMinDoms}
 \end{figure}

We now consider the differential that connect the various other
generator types. There are the following polygons:
\[\begin{tiny}\begin{CD}
  \Ax @>{\delta}>{\cald_4}> \BB,\qquad
  \Ax @>{L_1}>{\cald_2}> \XX,\qquad
  \XX_2 @>{L_2}>{\cald_3+\cald_6}>\BB_1\qquad
  \Ax_1 @>{R_2}>{\cald_5}>\Ax_2 \\
  \Ax_2 @>{(L_2,L_1)}>{\cald_2+\cald_6}>\YY_0, \qquad
  \XX_0 @>{(R_1,R_2)}>{\cald_1+\cald_5}> \XX_2, \qquad
  \YY_0 @>{(R_1,R_2)}>{\cald_1+\cald_5}> \YY_2, \qquad
  \Ax_1@>{(U_1,R_2)}>{\cald_1+\cald_2+\cald_5}>\XX_2 
\end{CD}\end{tiny}\]
We also have the following annuli that have holomorphic representatives:
\[\begin{tiny}\begin{CD}
  \Ax_2 @>{(L_2,U_1)}>{\cald_1+\cald_2+\cald_3+\cald_6}> \BB_1,\qquad
  \Ax_1 @>{(U_2,L_1)}>{\cald_2+\cald_5+\cald_6}> \YY_0,\qquad
  \XX_0 @>{(R_1,U_2)}>{\cald_1+\cald_3+\cald_5+\cald_6}>\BB_1 \\
  & \Ax_1@>{(U_1,U_2)}>{\cald_1+\cald_2+\cald_3+\cald_5+\cald_6}>\BB_1 \qquad
  \Ax_1@>{(U_2,U_1)}>{\cald_1+\cald_2+\cald_3+\cald_5+\cald_6}>\BB_1
\end{CD}
\end{tiny}
\]

The latter two relations force also 
another arrow
\[ \begin{CD}
  \Ax_1@>{\delta}>{\cald_1+\cald_2+\cald_3+\cald_5+\cald_6}>\BB_1 
\end{CD}\] which cancels against the $\delta$ induced by the bigon
$\cald_4$.  To see how, we consider the one-dimensional moduli space
${\mathcal M}$ of holomorphic maps from $\Ax_1$ to $\BB_1$ 
with the following properties:
\begin{itemize}
\item 
  The map has shadow
$\cald_1+\cald_2+\cald_3+\cald_5+\cald_6$.
\item The source has  two punctures
$p_1$ and $p_2$, marked by the Reeb orbit $\orb_1$ and a length one Reeb chord covering $\Zin_2$ respectively.
\item $t(p_1)>t(p_2)$.
\end{itemize}
The space ${\mathcal M}$ has an end which corresponds to the arrow
$\Ax_1$ to $\BB_1$ with algebra element $(U_1,U_2)$ (i.e. where $p_1$
goes to the boundary). In principle, it could have another end which
is a broken flowline. One of those components must contain a puncture
$p_1$ marked by $\orb_1$ (and no other punctures).  But there
evidently possible such shadow in the picture. It follows that ${\mathcal M}$ 
has an odd number of ends where $t(p_1)=t(p_2)$. But these are exactly
terms counted in $\delta^1_1$.

We choose an almost-complex structure which is sufficiently stretched
out normal to $\beta_2$. For such a choice of 
almost-complex structures, the analogous annuli
also have holomorphic representatives:
\[\begin{tiny}\begin{CD}
  \lsup{0}\Ax_{1,2} @>{(U_2,L_1)}>{\cald_2+\cald_5+\cald_6}> \lsup{0}\YY_{0,2},\qquad
  \lsup{0}\XX_{0,2} @>{(R_1,U_2)}>{\cald_1+\cald_3+\cald_5+\cald_6}>\lsup{0}\BB_{1,2} \\
\lsup{0}\Ax_{1,2}@>{(U_1,U_2)}>{\cald_1+\cald_2+\cald_3+\cald_5+\cald_6}>\BB_{1,2} \qquad
  \lsup{0}\Ax_{1,2}@>{(U_2,U_1)}>{\cald_1+\cald_2+\cald_3+\cald_5+\cald_6}>\lsup{0}\BB_{1,2}.
\end{CD}
\end{tiny}
\]
which also forces
\[ \begin{CD}
  \lsup{0}\Ax_{1,2}@>{\delta}>{\cald_1+\cald_2+\cald_3+\cald_5+\cald_6}>\lsup{0}\BB_{1,2}
  \end{CD}\]

We have verified the differentials from Figure~\ref{fig:MarkedMinDoms}.
To verify that $\YY_2$ and $\BB_2$ represent a submodule, we argue as
follows.  First, note that there are no positive domains that leave
the generator type $\BB$; and indeed there are no positive domains
leaving $\BB_2$ going to $\BB_1$. Moreover, a positive domain leaving $\YY$
must contain $\cald_3$.  It must then either be simply $\cald_3$, in
which case it represents the displayed differential; or it must also
contain $\cald_1$, in which case it starts at $\YY_0$ (rather than
$\YY_2$). This verifies that $\YY_2$ and $\BB_2$ represents a submodule.
We have seen that this submodule is acyclic. Thus, the entire module 
is homotopy equivalent to the quotient module (where we divide out by
$\YY_2$ and $\BB_2$.

We have now verified all the arrows in the quotient module of
Figure~\ref{fig:MarkedMinHatHol}. There are no other arrows. This can be
readily seen by the following: consider the periodic domain
$P=\cald_1+\cald_2+\cald_3-\cald_4+\cald_5+\cald_6$. This generates
the space of periodic domains locally. All possible domains connecting
generators are obtained from the shown domains plus some number of
copies of $P$. The result always has a negative local multiplicity
somewhere, except in one special case: both $\cald_4$ and
$\cald_1+\cald_2+\cald_3+\cald_5+\cald_6$ represent positive domains
from $\Ax$ to $\BB$. But these two domains have already been
considered above.

The stated homotopy equivalence follows easily.
\end{proof}

\subsection{The unblocked case}

\begin{proof}
  We have the following further bigons:
  \[
  \begin{CD}
  \BB@>{w z}>{\cald_w+\cald_z}>\Ax
  \qquad
  \YY@>{z}>{\cald_z}> \XX
  \end{CD}
  \]
  rectangles
  \[
  \begin{CD}
  \XX_2 @>{w}>{\cald_3+\cald_w}>
  \Ax_2
  &\qquad
  \Ax_2 @>{z\otimes L_2}>{\cald_6+\cald_z}>
  \Ax_1\\
    \BB_1 @>{w\otimes L_1}>{\cald_2+\cald_w}>\YY_0
    &\qquad \XX_0@>{w\otimes R_1}>{\cald_1+\cald_3+\cald_w}>\Ax_1
    \end{CD}
  \]
  annuli:
  \[
  \begin{CD}
  \Ax_1 @>{w\otimes U_1}>{\cald_1+\cald_2+\cald_3+\cald_w}> \Ax_1
  \qquad
  \lsup{0}\Ax_{1,2} @>{w\otimes U_1}>{\cald_1+\cald_2+\cald_3+\cald_w}> \lsup{0}\Ax_{1,2} \\
  \Ax_2@>{z\otimes U_2}>{\cald_5+\cald_6+\cald_z}> \Ax_2
  \qquad
  \lsup{0}\Ax_{1,2}@>{z\otimes U_2}>{\cald_5+\cald_6+\cald_z}> \lsup{0}\Ax_{1,2} \\
  \BB_1 @>{w\otimes U_1}>{\cald_1+\cald_2+\cald_3+\cald_w}> \BB_1
  \qquad
  \lsup{0}\BB_{1,2} @>{w\otimes U_1}>{\cald_1+\cald_2+\cald_3+\cald_w}> \lsup{0}\BB_{1,2} \\
  \BB_2@>{z\otimes U_2}>{\cald_5+\cald_6+\cald_z}> \BB_2
  \qquad
  \lsup{0}\BB_{1,2}@>{z\otimes U_2}>{\cald_5+\cald_6+\cald_z}> \lsup{0}\BB_{1,2} \\
  \XX_0 @>{w\otimes U_1}>{\cald_1+\cald_2+\cald_3+\cald_w}> \XX_0
  \qquad
  \lsup{0}\XX_{0,2} @>{w\otimes U_1}>{\cald_1+\cald_2+\cald_3+\cald_w}> \lsup{0}\XX_{0,2} \\
  \\
  \YY_0 @>{w}>{\cald_1+\cald_2+\cald_3+\cald_w}> \YY_0
  \qquad
  \lsup{0}\YY_{0,2} @>{w}>{\cald_1+\cald_2+\cald_3+\cald_w}> \lsup{0}\YY_{0,2} 
\end{CD}
  \]

  The bigon
  \[
  \begin{CD}
    \XX@>{w}>{\cald_4+\cald_w}>\YY
  \end{CD} \]
  cancels the annulus
  \[ \begin{CD}
  \XX@>{w}>{\cald_5+\cald_6}>\YY
  \end{CD}\]
  (For this last cancellation,
  observe that $\orb_2$ is an odd orbit according to our orbit marking;
  so although $\cald_5+\cald_6$ does not cross the $\wpt$ basepoint,
  it does contribute $w$, since the corresponding rectangle
  contains the orbit $\orb_2$.)

  Contracting the arrow from $\YY_2$ to $\BB_2$, we obtain an action
  corresponding to the juxtaposition of polygons
  \[ \begin{CD}
    \BB_1 @>{R_2}>{\cald_5}>\BB_2 @<{}<{\cald_3}< \YY_2 @>{z}>{\cald_z}>\XX_2.
    \end{CD} \]

  To summarize, we have verified the additional actions shown in
  Figure~\ref{fig:UnblockedActions} (i.e. in addition to the ones from
  Figure~\ref{fig:MarkedMinHatHol}).

  \begin{figure}
    \[ 
    \begin{tikzpicture}[scale=1.8]
    \node at (0,3) (A2) {$\Ax_1,~\lsup{0}\Ax_{1,2}$} ;
    \node at (-2.5,2) (A3) {$\Ax_2$} ;
    \node at (2.5,2) (A1) {$\XX_0,~\lsup{0}\XX_{0,2}$} ;
    \node at (-2.5,0) (B1) {$\YY_0,~\lsup{0}\YY_{0,2}$} ;
    \node at (2.5,0) (B3) {$\XX_2$} ;
    \node at (0,-1) (B2) {$\BB_1,~\lsup{0}\BB_{1,2}$} ;
    \draw[->] (A3) [bend right=5] to node[below,sloped] {\tiny{$z \otimes L_2$}}  (A2)  ;
    \draw[->] (A1) [bend left=5] to node[below,sloped] {\tiny{$w \otimes  R_1 $}}  (A2)  ;
    \draw[->] (B2) [bend left=5] to node[below,sloped,pos=.2]{\tiny{$w z$}} (A2);
    \draw[->] (B2) [bend right=5] to node[below,sloped] {\tiny{$z \otimes R_2 $}}  (B3)  ;
    \draw[->] (B2) [bend left=5] to node[below,sloped]{\tiny{$w\otimes L_1$}} (B1);
    \draw[->] (B3) [bend right=15] to node[below,sloped,pos=.8] {\tiny{$w$}} (A3) ;
    \draw[->] (B1) [bend right=15] to node[below,sloped,pos=.2] {\tiny{$z$}} (A1) ;
    \draw[->] (A2) [loop above] to node[above,sloped] {\tiny{$w\otimes U_1+ z \otimes U_2$}} (A2);
    \draw[->] (B2) [loop below] to node[below,sloped] {\tiny{$w\otimes U_1+ z \otimes U_2$}} (B2);
    \draw[->] (A1) [loop above] to node[above,sloped] {\tiny{$w \otimes U_1$}} (A2);
    \draw[->] (B1) [loop below] to node[below,sloped] {\tiny{$w \otimes U_1$}} (B1);
    \draw[->] (A3) [loop above] to node[above,sloped] {\tiny{$z \otimes U_2$}} (A3);
    \draw[->] (B3) [loop below] to node[below,sloped] {\tiny{$z \otimes U_2$}} (B3);
    \end{tikzpicture}\]
    \caption{Further actions}\label{fig:UnblockedActions}
    \end{figure}
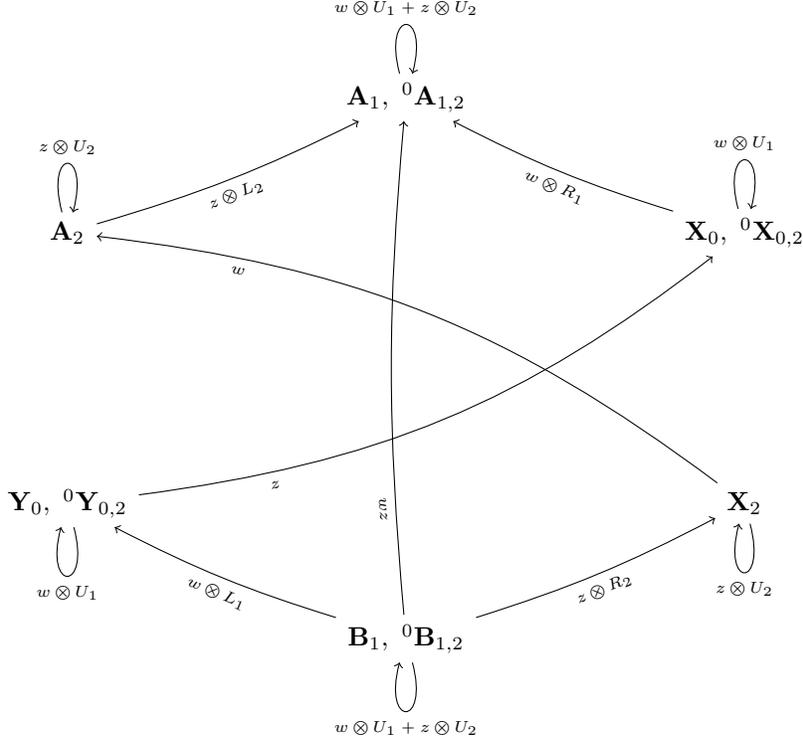

  In fact, the annuli illustrated above also demonstrate that there are
  no other self-arrows in the picture.
  For example, there is no $w\otimes U_1$ action from
  $\Ax_2$ to itself, although $\cald_1+\cald_2+\cald_3+\cald_4$ gives a
  domain. The reason for this is that there is no corner of $\Ax_2$ from
  which we can cut in to $\Zin_1$ to get a $U_1$ action.

  To get the stated identification 
  \[ \lsup{\cBlg_2}\DAmod(\HmidEx)_{\cBlg_1}\DT \lsup{\cBlg_1,\nDuAlg_1}\CanonDD
  \cong
  \lsup{\cBlg_2}\MarkedMin_{\cBlg_1}\DT \lsup{\cBlg_1,\nDuAlg_1}\CanonDD,\]
  we must argue that there are no further arrows in
  $\lsup{\cBlg_2}\DAmod(\HmidEx)_{\cBlg_1}\DT \lsup{\cBlg_1,\nDuAlg_1}\CanonDD$.
  It helps to observe that we already have no further arrows
  when we set $w=z=0$; and there are no additional self-arrows.

  The only other additional arrows that could be allowed by grading
  considerations. Consider $\lsup{\cBlg_2}\DAmod(\HmidEx)_{\cBlg_1}\DT~
  \lsup{\cBlg_1,\nDuAlg_1}\CanonDD$.  The relative Alexander gradings
  are specified by
  \[
    \Agr(\Ax_2)=\Agr(\YY_0)=\OneHalf \qquad
    \Agr(\Ax_1)=\Agr(\BB_1)=0 \qquad
    \Agr(\XX_0)=\Agr(\XX_2)=-\OneHalf,
    \]
  and $\Delta$-gradings by:
  \[ 
    \Mgr(\Ax_1)=\OneHalf \qquad
    \Mgr(\Ax_2)=\Mgr(\YY_0)=\Mgr(\XX_0)=\Mgr(\XX_2)=0
    \qquad 
    \Mgr(\BB_1)=-\OneHalf.
    \]
    Note also that $\Agr(w)=-1$, $\Mgr(w)=-1$; $\Agr(z)=1$,
    $\Mgr(z)=-1$.  Thus, the gradings and idempotents could allow
    $w\YY_0$ to appear in $\delta^1(\XX_0)$ and $z \XX_2$ to appear in
    $\delta^1(\Ax_2)$.  (Observe here that $\delta^1$ is taken with
    respect to the type $DD$ bimodules
    $\lsup{\cBlg_2}\DAmod(\HmidEx)_{\cBlg_1}\DT
    \lsup{\cBlg_1,\nDuAlg_1}\CanonDD$;
    whose generators we simply identify with the generators of 
    $\lsup{\cBlg_2}\DAmod(\HmidEx)_{\cBlg_1}$.)

    But both possibilities are excluded by $\delta^1\circ \delta^1=0$.
  
\end{proof}

\section{Computing link Floer homology}
\label{sec:ComputeHFL}

Recall that if $\DiagUp$ is an upper knot diagram with associated
curved algebra $\cClg$, in~\cite{HolKnot}, we defined a curved type
$D$ structure $\lsup{\cClg}\DmodAlg(\DiagUp)$. This was defined by
slicing up the diagram $\DiagUp$ into elementary pieces (crossings,
cups, and caps), associating algebraically defined curved type $DA$
bimodules (over $\Clg$) to all of those pieces, and then tensoring
them together.  It was then proved
in~\cite[Theorem~\ref{HK:thm:ComputeD2}]{HolKnot} that if $\DiagUp$ is
an upper diagram in bridge position and $\Hup$ is its associated
Heegaard diagram, then
\[ \lsup{\cClg}\Dmod(\Hup)\simeq
~\lsup{\cClg}\DmodAlg(\DiagUp).\]

In fact, the algebraic bimodules from~\cite{HolKnot} can all be
extended to $\cBlg$.  Bearing this in mind, we have the following
alternative characterization of $\lsup{\cClg}\DmodAlg(\DiagUp)$. Tensor
the algebraic bimodules over $\cBlg$ to arrive at a type $D$ structure
over $\Blg$, which we will denote $\lsup{\cBlg}\DmodAlg(\DiagUp)$. Then,
$\lsup{\cClg}\DmodAlg(\DiagUp)$ is the type $D$ structure
characterized by the property that
\[ \lsup{\cBlg}\DmodAlg(\DiagUp)\simeq~
\lsup{\cBlg}\iota_{\cClg}\DT~\lsup{\cClg} \DmodAlg(\DiagUp). \]

We can now extend the definition of $\lsup{\cBlg}\DmodAlg(\DiagUp)$ to
canonically marked upper diagrams, by declaring the algebraic bimodule
(over $\cBlg$) associated to a marked minimum to be the bimodule
defined in Section~\ref{sec:AlgMarkedMin}.

\begin{thm}
  \label{thm:ComputeMarkedUpper}
  Let $\Hup$ be a canonically marked upper diagram in bridge position
  and $\Hup$ be its associated marked upper Heegaard diagram,
  then there is a type $D$ structure
  $\lsup{\cClg}\DmodAlg(\DiagUp)$ uniquely characterized by the property that
  \[ \lsup{\cBlg}\DmodAlg(\DiagUp)\simeq~ \lsup{\cBlg}\iota_{\cClg}\DT~
  \lsup{\cClg}\DmodAlg(\DiagUp). \] Moreover, there is an identification
  \[ \lsup{\cClg}\Dmod(\Hup)\simeq~\lsup{\cClg}\DmodAlg(\DiagUp).\]
\end{thm}

The above result will be seen as a consequence of the following
proposition.
As in~\cite{HolKnot},
a type $D$ module $\lsup{\cClg}X$ is called {\em relevant} if
there is an $\Ainfty$ bimodule $Z_{\nDuAlg}$ with the property that
\[ \lsup{\cBlg}\iota_{\cClg}\DT~ \lsup{\cClg}X
\simeq Z_{\nDuAlg} \DT \lsup{\nDuAlg,\cBlg}\CanonDD.\]

\begin{prop}
  \label{prop:ExtendByOne}
  If $\lsup{\cClg_1}X$ is relevant, then 
  $\lsup{\cClg_2}Y=\lsup{\cClg_2}\DAmod(\Hmid)_{\cClg_1}\DT\lsup{\cClg_1}X$
  is uniquely characterized up to homotopy by the property that
  \begin{equation}
    \label{eq:ComputeY}
    \lsup{\cBlg_2}\iota_{\cClg_2}\DT~\lsup{\cClg_2}Y\simeq 
    \lsup{\cBlg_2}\MarkedMin_{\cBlg_1}\DT~ \lsup{\cBlg_1}\iota_{\cClg_1}\DT \lsup{\cClg_1} Y.
  \end{equation}
  Moreover, $\lsup{\cClg_2}Y$ is relevant.
\end{prop}

\begin{proof}
  First, we verify Equation~\eqref{eq:ComputeY} (using, in order, the
  definition of $Y$, relevance of $X$, commutativity of $\DT$,
  Proposition~\ref{prop:ComputeDD}, and the relevance of $X$):
  \begin{align*}
    \lsup{\cBlg_2}\iota_{\cClg_2}\DT~ \lsup{\cClg}Y&=
    \lsup{\cBlg_2} \iota_{\cClg_2}\DT~
    \lsup{\cClg_2}\DAmod(\Hmid)_{\cClg_1}\DT~ \lsup{\cClg_1} X\\
    &\simeq
    \lsup{\cBlg_2}\DAmodEx(\HmidEx)_{\cBlg_1}\DT~
    \lsup{\cBlg_2} \iota_{\cClg_2}\DT~\lsup{\cClg_1} X \\
    &\simeq
    \lsup{\cBlg_2}\DAmodEx(\HmidEx)_{\cBlg_1}\DT~
    (Z_{\nDuAlg_1}\DT~
    \lsup{\cBlg_1,\nDuAlg_1}\CanonDD) \\
    &\simeq
    Z_{\nDuAlg_1}\DT
    (\lsup{\cBlg_2}\DAmodEx(\HmidEx)_{\cBlg_1}\DT~
    \lsup{\cBlg_1,\nDuAlg_1}\CanonDD) \\
    &\simeq
    Z_{\nDuAlg_1}\DT
    (\lsup{\cBlg_2}\MarkedMin_{\cBlg_1}\DT~
    \lsup{\cBlg_1,\nDuAlg_1}\CanonDD) \\
    &\simeq
    \lsup{\cBlg_2}\MarkedMin_{\cBlg_1}\DT~
    (Z_{\nDuAlg_1}\DT~
    \lsup{\cBlg_1,\nDuAlg_1}\CanonDD) \\
    &\simeq
    \lsup{\cBlg_2}\MarkedMin_{\cBlg_1}\DT~
    \lsup{\cBlg_2} \iota_{\cClg_2}\DT~\lsup{\cClg_1} X.
  \end{align*}

  Equation~\eqref{eq:ComputeY} uniquely characterizes
  $\lsup{\Clg_2}Y$, according to Lemma~\ref{lem:Injection}. 

  To see $\lsup{\cClg}Y$ is relevant, we use the above computation
  and Proposition~\ref{prop:CommuteWithCanon}, as follows:
  \begin{align*}
    \lsup{\cBlg_2}\iota_{\cClg_2}\DT~ \lsup{\cClg}Y
    &
    \simeq 
    Z_{\nDuAlg_1}\DT
    (\lsup{\cBlg_2}\MarkedMin_{\cBlg_1}\DT~
    \lsup{\cBlg_1,\nDuAlg_1}\CanonDD) \\
    &
    \simeq 
    Z_{\nDuAlg_1}\DT
    (\lsup{\nDuAlg_1}\nMarkedMin_{\nDuAlg_2}\DT~
    \lsup{\nDuAlg_2,\cBlg_2}\CanonDD) 
    =
    (Z_{\nDuAlg_1}\DT
    \lsup{\nDuAlg_1}\nMarkedMin_{\nDuAlg_2})\DT~
    \lsup{\nDuAlg_2,\cBlg_2}\CanonDD 
    \end{align*}
\end{proof}

\begin{proof}[Proof of Theorem~\ref{thm:ComputeMarkedUpper}]
  We prove by induction on the number of components
  that $\lsup{\cClg}\DmodAlg(\DiagUp)$ is relevant.
  The case of one component was proved in~\cite{HolKnot}.
  The inductive hypothesis then follows from
  Proposition~\ref{prop:ExtendByOne}.
\end{proof}

Let $\Diag$ be a canonically marked planar diagram for
an oriented link, and let
$\DiagUp$ be the upper diagram obtained by removing the
global minimum, and $\Hup$ be its associated upper diagram.
The output algebra $\Clg(1)$ of $\DiagUp$ is identified with
\[ \Field[U_1,U_2]/U_1 U_2.\]
There is a ring  isomorphism
\[ \Psi\colon \Clg(1)\to \Field[w_1,z_1]/w_1 z_1,\]
sending $U_1$ to $w_1$ and $U_2$ to $z_1$.

\begin{thm}
  \label{thm:GeneralComputation}
  For any Heegaard diagram $\HD$ representing $\orL^+$,
  there is a chain homotopy equivalence
  \[\CFL(\HD)\simeq [\Psi]_{\Clg(1)}\DT \DmodAlg(\DiagUp).\]
\end{thm}

\begin{proof}
  The bimodule $\lsup{\Field[w,z]/wz}[\Psi]_{\cClg(1)}$ can
  be thought of as the type $A$ module associated to the global
  minimum. (See~\cite[Lemma~\ref{HK:lem:GlobalMinimum}]{HolKnot}.
  Thus, the lemma can be seen as a consequence of (a particular easy
  special case of) the pairing theorem.
\end{proof}

\begin{proof}[Proof of Theorem~\ref{thm:ComputeHFLa}]
  This follows from Theorem~\ref{thm:GeneralComputation},
  by specializing to 
  \[ w_1=z_1=\dots=w_\ell=z_\ell=0.\]
\end{proof}

\subsection{Restricting idempotents}

In our computations, we do not need the full bimodule
$\lsup{\cBlg_2}\MarkedMin_{\cBlg_1}$, as we are able to work over $\cClg$.
Thus, for computations, it suffices to use only
$\lsup{\cBlg_2}\MarkedMin_{\cBlg_1}\DT~\lsup{\cBlg_1}\iota_{\cClg_1}$.
The actions for this bimodule are described in Figure~\ref{fig:MarkedMinOverC};
see also Figure~\ref{fig:MarkedMinOverCHat}
for 
$\lsup{\cBlg_2}\MarkedMinHat_{\cBlg_1}\DT~\lsup{\cBlg_1}\iota_{\cClg_1}$.
\[ \lsup{\Blg(n,\Matching')}\MarkedMin_{\Clg(n+1,\Matching)}
= \lsup{\Blg(n,\Matching')}\MarkedMin_{\Blg(n+1,\Matching)}
\DT \lsup{\Blg(n+1)}i_{\Clg(n+1)}:\]

\begin{figure}
     \begin{tikzpicture}[scale=1.8]
     \node at (0,5) (A) {$A_1$} ;
     \node at (-2,3) (Y) {$Y_1$} ;
     \node at (2,3) (X) {$X_1$} ;
     \node at (0,2) (B) {$B_1$} ;
     \draw[->] (A)  [bend left=5] to node[below,sloped] {\tiny{$z^\ell\otimes R_2 U_2^\ell$}}  (Y)  ;
     \draw[->] (Y)  [bend left=5] to node[above,sloped] {\tiny{$z^{\ell+1}\otimes L_2 U_2^\ell$}}  (A)  ;
     \draw[->] (X) [loop right] to node[above,sloped]{\tiny{$z^\ell\otimes U_2^\ell$}} (X);
     \draw[->] (X)  [bend left=5] to node[below,sloped] {\tiny{$z^\ell\otimes L_2 U_2^\ell$}}  (B)  ;
     \draw[->] (B)  [bend left=5] to node[above,sloped] {\tiny{$z^{\ell+1}\otimes R_2 U_2^\ell$}}  (X)  ;
     \draw[->] (A) [bend left=5] to node[above,sloped] {\tiny{$w^{k}z^\ell\otimes (U_1^{k+1},R_2U_2^\ell)$}}  (X)  ;
     \draw[->] (Y)  to node[below,sloped] {\tiny{$w^{k} z^{\ell}\otimes (L_2 U_2^{\ell},U_1^{k+1})$}}  (B)  ;
     \draw[->] (A) [bend left=5] to node[above,sloped] {\tiny{$w^{k}z^\ell\otimes ((U_1^{k+1},U_2^{\ell+1})+(U_2^{\ell+1},U_1^{k+1}))$}}  (B)  ;
 \draw[->](B) [bend left=5] to node[above,sloped]{\tiny{$wz$}} (A) ;
     \draw[->] (A) [loop above] to node[above,sloped]{\tiny{$w^k z^\ell \otimes U_1^k U_2^{\ell}$}} (A);
     \draw[->] (B) [loop below] to node[below,sloped]{\tiny{$w^k z^\ell \otimes U_1^k U_2^{\ell}$}} (B);
     \draw[->] (X) to node[below,sloped,pos=.8] {\tiny{$w$}} (Y) ;
     \draw[->] (Y) [loop left] to node[above,sloped]{\tiny{$z^\ell\otimes U_2^\ell$}} (Y);
   \end{tikzpicture}
   \caption{\bf Actions for $\lsup{\cBlg_2}\MarkedMin_{\cBlg_1}\DT \lsup{\cBlg_1}\iota_{\cClg_1}$}
   \label{fig:MarkedMinOverC}
\end{figure}

\begin{figure}
    \begin{tikzpicture}[scale=1.8]
    \node at (0,4) (A) {$A_1$} ;
    \node at (-1,3) (Y) {$Y_2$} ;
    \node at (1,3) (X) {$X_2$} ;
    \node at (0,2) (B) {$B_1$} ;
    \draw[->] (A) to node[above,sloped] {\tiny{$R_2$}}  (Y)  ;
    \draw[->] (X) to node[below,sloped] {\tiny{$L_2$}}  (B)  ;
    \draw[->] (A) to node[above,sloped] {\tiny{$(U_1,R_2)$}}  (X)  ;
    \draw[->] (Y)  to node[below,sloped] {\tiny{$(L_2,U_1)$}}  (B)  ;
    \draw[->] (A)  to node[above,sloped] {\tiny{$(U_1,U_2)+(U_2,U_1)$}}  (B)  ;
  \end{tikzpicture}
   \caption{\bf Actions for $\lsup{\cBlg_2}\MarkedMinHat_{\cBlg_1}\DT \lsup{\cBlg_1}\iota_{\cClg_1}$}
   \label{fig:MarkedMinOverCHat}
\end{figure}

\bibliographystyle{plain}
\bibliography{biblio}

\end{document}